\documentclass{article}
\usepackage{graphicx} 

\usepackage{amscd}
\usepackage{amssymb}
\usepackage{amsmath}
\usepackage{mathrsfs}
\usepackage[all,cmtip]{xy}
\usepackage{geometry}
\usepackage{longtable}

\usepackage{amsthm}
\usepackage{tikz}
\usepackage{hyperref}

\newtheorem{thm}{Theorem}[section]
\newtheorem{cor}[thm]{Corollary}
\newtheorem{conj}[thm]{Conjecture}
\newtheorem{lemma}[thm]{Lemma}
\newtheorem{prop}[thm]{Proposition}

\theoremstyle{definition}
\newtheorem{definition}[thm]{Definition}
\theoremstyle{remark}
\newtheorem{remark}[thm]{Remark}
\newtheorem{example}[thm]{Example}

\theoremstyle{definition}
\newtheorem{thmalt}{Theorem}[thm]

\newenvironment{thmp}[1]{
  \renewcommand\thethmalt{#1}
  \thmalt
}{\endthmalt}

\usepackage{afterpage}
\usepackage[utf8]{inputenc}
\usepackage{tikz-cd}
\usetikzlibrary{arrows}
\usepackage{epsfig,epstopdf,titling,url,array}
\usepackage{float}
\usepackage{caption}
\usepackage{lipsum}
\usetikzlibrary{calc}
\usepackage{mathtools,leftindex,tensor}
\usepackage[version=4]{mhchem}
\usetikzlibrary{decorations.pathreplacing,calligraphy}
\usepackage{ulem}

\def\cf{\mathfrak{c}}
\def\mf{\mathfrak{m}}

\def\gf{\mathfrak{g}}

\def\Cbb{\mathbb{C}}

\def\Rbb{\mathbb{R}}

\def\Zbb{\mathbb{Z}}
\def\Gbb{\mathbb{G}}

\def\Fbb{\mathbb{F}}

\def\Abb{\mathbb{A}}

\def\Oc{\mathcal{O}}

\def\Fc{\mathcal{F}}

\def\Nc{\mathcal{N}}
\def\Pc{\mathcal{P}}

\def\Cc{\mathcal{C}}

\def\id{\mathrm{id}}
\def\Hom{\operatorname{Hom}}
\def\ker{\operatorname{ker}}

\def\Ind{\operatorname{Ind}}

\def\End{\operatorname{End}}
\def\Res{\operatorname{Res}}

\def\Ver{\operatorname{Ver}}
\def\Rep{\operatorname{Rep}}
\def\Tilt{\operatorname{Tilt}}
\def\dim{\operatorname{dim}}
\def\Vec{\operatorname{Vec}}
\def\sVec{\operatorname{sVec}}
\def\Gr{\operatorname{Gr}}
\def\ol{\overline}
\def\Spec{\operatorname{Spec}}
\def\Dist{\operatorname{Dist}}
\def\ext{\operatorname{ext}}
\def\sub{\subseteq}

\usepackage{authblk}

\title{A Remarkable Functor on $G$-Modules}
\author[1,2]{Joseph Baine}
\author[1,2]{Tasman Fell}
\author[2]{Anna Romanov}
\author[1,2]{Alexander Sherman}
\author[1]{Geordie Williamson}
\affil[1]{University of Sydney}
\affil[2]{University of New South Wales}

\date{}

\begin{document}
\normalem

\maketitle

\begin{abstract}

    We introduce a new functor on categories of modular representations of reductive algebraic groups. 
    Our functor has remarkable properties. 
    For example it is a symmetric monoidal functor and sends every standard and costandard object in the principal block to a one-dimensional object. 
    We connect this new functor to recent work of Gruber and
    conjecture that it is isomorphic to hypercohomology
    under the equivalence of the Finkelberg-Mirkovi\'c conjecture.

\end{abstract}


\section{Introduction}

	The modular representation theory of algebraic groups is an astonishingly rich and deep subject, with influences from compact Lie groups (Weyl’s character formula), finite groups (decomposition numbers and Brauer reciprocity) and geometric representation theory (Kazhdan-Lusztig and Lusztig conjectures). 
	More recently, striking connections to algebraic number theory have emerged (see e.g. \cite{LLLM}).

	In this paper we introduce a new functor on the category of representations of a reductive algebraic group. 
	Our functor is simple; it is given by restricting to the first Frobenius kernel of a regular unipotent subgroup, and throwing away all projective summands.  
	We prove that this functor has remarkable properties: it is symmetric monoidal, and sends any standard or costandard module in the principal block to a one-dimensional object. 
	It is crucial for our arguments that the target of our functor is the Verlinde category, an exotic symmetric tensor category whose characteristic 0 analogue is only braided, but not symmetric.  
	We were  motivated by work of Duflo and Serganova  \cite{duflo2005associated} who introduced a similar functor to super vector spaces in the setting of Lie superalgebras. 
	See also the parallel work \cite{hone2025semisimplifying} where a similar functor was studied on modular representations of symmetric groups.

	Our functor should allow detailed study of tensor products of modular representations. 
	In particular, we prove that it is well-adapted to the study of Gruber’s regular modules \cite{Gruber}. 
	We also conjecture that it has an alternative description in the language of geometric representation theory. 
	Namely, it should provide an algebraic incarnation of hypercohomology under a conjectural equivalence due to Finkelberg and Mirkovi\'c  \cite{FM} which has recently been established by Bezrukavnikov and Riche \cite{BR2}.  

\subsection{Main results}

	Let $\Bbbk$ be an algebraically closed field of characteristic $p>0$, and let $G$ be a reductive algebraic group over $\Bbbk$ (for precise assumptions see \S\ref{sec: Alg grps}). 
	Fix a maximal torus and Borel subgroup $T \subset B \subset G$, and let $\mathfrak{X}_+ \subset \mathfrak{X}$ denote the (dominant) weights of $T$ determined by the opposite Borel to $B$. 
	We consider the category $\Rep (G)$ of algebraic representations of $G$.
	Examples of representations in $\Rep (G)$ include the standard, costandard, simple and indecomposable tilting modules of highest weight $\lambda \in \mathfrak{X}_+$, that we denote $\Delta_\lambda, \nabla_\lambda, L_\lambda$ and $T_\lambda$, respectively.

	Let $K \subset G$ denote a principal $SL_2$ subgroup\footnote{Although it is common to refer to this subgroup as the ``principal $SL_2$ subgroup'' the reader should keep in mind that this terminology is slightly deceptive. It is a rank 1 subgroup, isomorphic to either $SL_2$ or $PGL_2$. The latter case occurs, e.g. in $SL_3$.} such that $T_K = T \cap K$ is a maximal torus in $K$ and $B_K = B \cap K$ is a Borel subgroup of $K$. 
	Let $H$ denote the first Frobenius kernel of $U_K$, the unipotent radical of $B_K$. 
	Thus, $H$ is a regular unipotent group scheme isomorphic to $\alpha_p$, the first Frobenius kernel of the additive group $\mathbb{G}_a$. 

	Recall that representations of $\alpha_p$ are easily described: they are simply vector spaces together with a nilpotent endomorphism of degree at most $p$.  
	The Verlinde category is formed by semisimplifying the category of representations of $\alpha_p$.  
	One obtains in this way a semisimple symmetric tensor category $\Ver_p$ with $p-1$ simple objects.  
	The Verlinde category has an (a priori surprising) symmetry, given by tensoring with the $(p-1)$-dimensional Jordan block, whenever $p > 2$. 
	We denote this functor by $\Pi$ (often called `parity shift'), and note that $\Bbbk$ and $\Pi \Bbbk$ give a tensor subcategory of $\Ver_p$ isomorphic to the tensor category of super vector spaces. 
	We denote this tensor subcategory by $\operatorname{sVec}\subset \Ver_p$.

	Consider the functor
	\[
		\Phi_H : \Rep (G) \to \Ver_p,
	\]
	given by first restricting to $H$, and then taking the image in the Verlinde category under semisimplification. 
	(Note that $\Phi_H$ has an explicit, elementary description, which is explained in \S\ref{sec: OTI}.)  
	This defines a symmetric monoidal functor which is not exact in general. 
	Our first main theorem is that the value of $\Phi_H$ on (co)standard modules is remarkably simple, and is controlled by the extended affine Weyl group under the $p$-dilated dot action.

	\begin{thmp}{A}
	\label{thm A} 
		Let $\lambda,\mu\in\mathfrak{X}_+$, and take $s\in W$ (the affine Weyl group) to be a reflection such that $s\cdot\lambda\in\mathfrak{X}_+$. 
		Then we have natural isomorphisms
    	\begin{enumerate} 
        	\item $\Phi_{H}(\Delta_\lambda)\cong\Pi\Phi_{H}(\Delta_{s \cdot \lambda})$,
        	\item $\Phi_{H}(\Delta_\lambda) \cong\Phi_{H}(\Delta_{\lambda+p\mu}).$
    	\end{enumerate}
	\end{thmp}

	Note that we prove Theorem \ref{thm A} for costandard modules in the body of the paper, which is equivalent because $\Phi_H$ is symmetric monoidal and we have $\Delta_\lambda^*\cong\nabla_{-w_0(\lambda)}$, where $w_0$ denotes the longest element of the finite Weyl group.

	Let $W^{\ext}$ denote the extended affine Weyl group. 
	Both $W$ and $W^{\ext}$ act on $\mathfrak{X}$ via the $p$-dilated dot action.  
	As a consequence of the linkage principle, we have the block decomposition
	\[
		\Rep (G) 
		= 
		\bigoplus_{ [\gamma] \in \mathfrak{X}/(W\cdot)} \Rep_\gamma (G) 
		\quad 
		\text{where} 
		\quad 
		\Rep_\gamma (G) = \langle L_\lambda 
		\; | \; 
		\lambda \in W \cdot \gamma \cap \mathfrak{X}_+ \rangle.
	\]
	Here we write $[\gamma]$ for the coset containing $\gamma$.  
	We consider the principal block $\Rep_0(G)$ as well as the extended principal block
	\[
		\Rep_0^{\ext} (G) 
		= 
		\langle L_\lambda 
		\; | \; 
		\lambda \in W^{\ext} \cdot 0 \cap \mathfrak{X}_+ \rangle.
	\]

	Theorem A implies that all (co)standard modules in the extended principal block are mapped to either $\Bbbk$ or $\Pi \Bbbk$ under $\Phi_H$. 
	Our second theorem shows that, in fact, our functor maps the entire extended principal block to super vector spaces:

	\begin{thmp}{B}
    	\[
    	\Phi_H(\Rep_0^{\ext}(G))\subseteq\operatorname{sVec}.
    	\]
	\end{thmp}

	One may unpack Theorems A and B into a concrete statement about (co)standard modules in the extended principal block which is rather striking:

	\begin{cor}
	\label{cor yet another cor}  
		Let $H\sub G$ be as above.
   		\begin{enumerate}
        	\item The restriction to $H$ of a (co)standard module in the extended principal block has a unique Jordan block of dimension less than $p$.
        	\item The restriction to $H$ of any module in the extended principal block has all Jordan blocks of dimensions $1$, $p-1$ and $p$.
    	\end{enumerate}
	\end{cor}

	\begin{remark}
    	In general, the special Jordan block in (1) of Corollary \ref{cor yet another cor} sits rather non-trivially inside the (co)standard module.
    	For example, outside of $SL_2$, it typically has no intersection with the highest and lowest weight spaces.
	\end{remark}

	Recall that our functor $\Phi_H$ is given by restriction to $H$, followed by a semisimplification procedure. 
	As such, it has extra structure given by the action of the centraliser $C_G(H)$ of $H$ in $G$. 
	Thus, we may view $\Phi_H$ as a functor
	\[
	\Phi_H : \Rep (G) \to \Rep_{\Ver_p}(C_G(H)). 
	\]
	The group $C_G(H)$ is the centraliser in $G$ of a regular nilpotent element in the Lie algebra, and it has a beautiful structure that was studied by Steinberg (\cite{SteinbergConjClasses}), Springer (\cite{springer}), Kostant (\cite{KostantPolynomial},\cite{KostantTDS}), and more recently by Yun-Zhu (\cite{YunZhu}) and Bezrukavnikov-Riche-Rider (\cite{BRR}).
	Theorems of Ginzburg and Yun-Zhu provide a homological interpretation of this centraliser: its coordinate ring is isomorphic to the homology of the affine Grassmannian of the Langlands dual group of $G$.

	One may incorporate the action of the normaliser of $H$, $N_G(H)$, in order to introduce a grading.  
	One recovers in this way the grading on homology, under the isomorphism of the coordinate ring with the homology of the affine Grassmannian. 
	We may then upgrade $\Phi_H$ to a functor taking values in graded modules over $C_G(H)$. 
	Restricting this picture to the extended principal block, Theorem B implies that we have a functor:
	\[
	\Phi_H:\Rep^{\ext}_0(G)\to\operatorname{gr}_{\Zbb}\Rep_{\operatorname{sVec}}(C_G(H))
    .
	\]
	We show the action of $C_G(H)$ on $\Phi_H(\Rep^{\ext}_0(G))$ factors over its Frobenius twist. 
	We conjecture (Conjecture \ref{conj:FM}) that our functor has an incarnation in terms of constructible sheaves: it should be isomorphic to the hypercohomology functor under an equivalence conjectured by Finkelberg and Mirkovi\'c, and recently proved by Bezrukavnikov and Riche.

	\begin{remark}
    	It is natural to ask whether the above picture still works if we instead take \linebreak $H\cong C_p=U_K(\mathbb{F}_p)\sub G$.  
    	It is known that the semisimplification of $\Rep (C_p)$ is also $\operatorname{Ver}_p$.  
    	We show in \S3 that Theorem A also holds in this case; however, Theorem B fails.  
    	Already for $SL_2$ we have that $\Phi_H(L_{2p-2})$ is not a super vector space.  
    	One essential difference between these two cases is the compatibility with Frobenius twists.  
    	Indeed, $\Phi_{C_p}(V^{(1)})\cong \Phi_{C_p}(V)^{(1)}$ while $\Phi_{\alpha_p}(V^{(1)})=V$. 
    	(Here we use the symbol $V$ to denote both the original $G$-module and its image in $\Ver_p$, under the inclusion $\operatorname{Vec}\sub\operatorname{Ver}_p$.)  
    	Another important difference is that the normaliser of $\alpha_p$ is much larger than that of $C_p$.
	\end{remark}

\subsection{Motivation from Lie superalgebras}

	The original motivation for studying the functor $\Phi_H$ came from the representation theory of Lie superalgebras over $\Cbb$.  
	Here, one of the most powerful tools is the Duflo-Serganova functor, introduced in \cite{duflo2005associated}.  
	It may be defined as follows: if $\gf$ is a Lie superalgebra and $x\in\gf$ is an odd element, we have that $\frac{1}{2}[x,x]=x^2$ in the universal enveloping algebra $\mathcal{U}\mathfrak{g}$.  
	In particular, the condition $[x,x]=0$ is nontrivial, and implies that $x$ acts by a square-zero operator on every representation.  
	Given such an $x$, the Duflo-Serganova functor $DS_x:\Rep (\gf)\to\operatorname{sVec}$ is given by the homology of the operator $x$.  
	An equivalent definition is obtained by the diagram:
	\[
	\xymatrix{\Rep(\gf)\ar[rr]^{Res}\ar[rrd]_{DS_x} & & \Rep(\Gbb^{0|1}) \ar[d]^{ss}\\ && \operatorname{sVec}}
	\]
	where $\Gbb^{0|1}$ is the purely odd additive supergroup of dimension $(0|1)$, whose representation theory is equivalent to modules over $\Cbb[x]/x^2$.  
	Here, the functor $ss$ denotes semisimplification. 
	We thus see a clear parallel to the functor $\Phi_H$ in positive characteristic.

	The Duflo-Serganova functor has been used in the study of blocks, central characters, superdimension formulae, categorical actions, and tensor products for Lie superalgebras.  
	For a survey of this functor and its applications, see \cite{gorelik2022duflo}.

\subsection{Relation to Gruber's work} 
	Theorem B gives us a tensor functor
	\[
		\Phi_H : \Rep_0^{\ext}(G) \to \sVec \subset \Ver_p.
	\]
	Given two modules $M, N$ in the (extended) principal block, their tensor product $M \otimes N$ (almost) never lies in the principal block. 
	On the other hand, the value of $\Phi_H$ on a general module in $\Rep (G)$ typically involves many summands in $\Ver_p$, and certainly there is no reason to suspect that it should land in $\operatorname{sVec} \subset \Ver_p$. 
	It is thus surprising that any summand of $M \otimes N$ will be mapped under $\Phi_H$ to $\operatorname{sVec}$.

	An explanation for this curious behaviour is provided by beautiful recent observations of Gruber \cite{Gruber}. 
	For any $M \in \Rep (G)$, Gruber considers a minimal complex $C^\bullet_M$ of tilting modules with cohomology $M$. 
	He calls $M$ \emph{singular} if every indecomposable summand of $C^i_M$ has dimension divisible by $p$ for all $i$.  
	This defines a thick tensor ideal, which we denote by $\Rep_{sing}$, whose objects are singular modules. 
	One may consider the quotient of additive categories (a tensor category)
	\[
		\underline{\Rep}(G) = \Rep (G) / \Rep_{sing}.
	\]
	Gruber proves that if we denote by $\underline{\Rep}^{\ext}_0(G)$ the image of the extended principal block in $\underline{\Rep}(G)$, then $\underline{\Rep}^{\ext}_0(G)$ is closed under tensor product. 
	In \S \ref{sec: Singular modules} we prove that $\Phi_H$ vanishes on $M$ if and only if $M$ is singular. 
	In particular, $\Phi_H$ factors over $\underline{\Rep}(G)$. 
	Gruber has since begun a systematic study of tensor product multiplicities in $\underline{\Rep}(G)$ \cite{GruberGeneric}. 
	We hope our functor $\Phi_H$ may provide a new tool in studying these questions.

\subsection{Relation to Finkelberg-Mirkovic conjecture} 
	The Finkelberg-Mirkovi\'c conjecture is one of the most useful guiding principles in the modular representation theory of algebraic groups (see e.g. \cite[\S 2.5]{WTakagi} or \cite[\S 13]{CW}). 
	It has recently been proven as the culmination of three deep works of Bezrukavnikov, Riche and Rider \cite{BRR, BR1, BR2}, as a consequence of a modular analogue of Bezrukavnikov's two realizations of the affine Hecke category \cite{Bez}.

	Let us briefly recall the statement of the Finkelberg-Mirkovi\'c conjecture, before pointing out the relevance to our work. 
	Let ${}^LG$ be the complex group which is dual in the sense of Langlands to $G$, and let $\Gr= {}^LG((t))/{}^LG[[t]]$ denote the affine Grassmannian for ${}^LG$.
	Let ${}^LB \subset {}^LG$ be the subgroup corresponding to a choice of Borel subgroup $B \subset G$, and let $\mathrm{Iw}$ denote the Iwahori subgroup of ${}^LG((t))$ corresponding to our choice of Borel ${}^LB \subset {}^LG$. 
	Finkelberg and Mirkovi\'c conjectured an equivalence \cite{FM}
	\begin{equation} 
	\label{eq:FM2}
		\Rep_0^{\ext}(G) \stackrel{\sim}{\to} P_{(\mathrm{Iw})}(\Gr,\Bbbk)
	\end{equation}
	where $P_{(\mathrm{Iw})}(\Gr,\Bbbk)$ denotes the category of perverse sheaves on $\Gr$ which are constructible with respect to the stratification by Iwahori orbits.\footnote{In \cite{BR2}, this is stated in terms of $\mathrm{Iw}_u$-equivariant sheaves, where $\mathrm{Iw}_u$ denotes the pro-unipotent radical of $\mathrm{Iw}$, however these two categories are equivalent.}

	As with many equivalences appearing in geometric Langlands duality, functors or operations on one side may be  mysterious on the other side.
	A particular instance of this is given by the (hyper)cohomology functor $H^*$ on $P_{(\mathrm{Iw})}(\Gr, \Bbbk)$, which produces graded modules over $H^*(\Gr)$. 
	Ever since the statement of the Finkelberg-Mirkovi\'c conjecture, it has been an intriguing problem to describe this functor on the other side of the equivalence. 
	This is a particularly appealing problem as the hypercohomology functor is central to other ``Soergel type" equivalences (see e.g. \cite{BGS,Ginzburg,SKoszul}).

	In the final section we gather evidence that $\Phi_H$ provides an algebraic incarnation of hypercohomology:

	\begin{conj} 
	\label{conj:FM}
		Under the identification of $H^*(\Gr) = \Dist C_G(e)^{(1)}$ we have a commuting diagram
		\[
		\begin{tikzcd}
    		\mathrm{Rep}_0^{\mathrm{ext}}(G) \arrow[r, "\sim"] \arrow[d, swap, "\Phi_H"] & P_{(\mathrm{Iw})}(\mathrm{Gr}, \Bbbk) \arrow[d, "H^*"] \\
  			\Dist C_G(e)^{(1)}\text{-}\operatorname{mod} \arrow[r, "\sim"]  & H^*(\mathrm{Gr})\text{-}\operatorname{mod}.
		\end{tikzcd}
		\]
		(Recall from the above discussion that $\Phi_H$ can be viewed as taking values in super vector spaces with $C_G(H)^{(1)}$-action. 
        Here we forget the $\mathbb{Z}/2\mathbb{Z}$-grading arising from the super vector space structure.)
    \end{conj} 

    We also give a version of this conjecture that incorporates $\mathbb{Z}$-gradings (again forgetting the super vector space structure) in Conjecture \ref{conj:FM2}.

	Establishing this conjecture could eventually lead to a simplified proof of the Finkelberg-Mirkovi\'c conjecture. 
	In any case, it seems important to understand the relation between $\Phi_H$ and the proof of the Finkelberg-Mirkovi\'c conjecture in \cite{BR2}.  
	As evidence for the conjecture, we prove that $\Phi_H$ is
        homological (see Remark \ref{rem:homological}) and agrees with the cohomology functor on tilting and (co)standard modules (\S \ref{sec: FM conj}). 

\subsection{Structure of this paper} 
	This paper is set up as follows:
	\begin{itemize}
		\item In \S 2 we collect preliminary facts and notation.
		\item We first prove Theorem A (\S\ref{Sec: Proof Thm A}). Here we use techniques from algebraic geometry and translation functors.
		\item We then prove Theorem B (\S \ref{Sec: Proof Thm B}). Here the methods are homological.
		\item In \S \ref{Sec: Further} we establish some other results discussed in the introduction. We reinterpret our functor in terms of minimal complexes of tilting modules, connect our functor to Gruber's theory and establish some results towards Conjecture \ref{conj:FM}.
	\end{itemize}

\subsection{Acknowledgements} 
	We thank Chris Hone, Finn Klein, Bregje Pauwels, Oded Yacobi, and Victor Zhang, the better half of the OTI team, without whose support and consultation this project would not have gotten off the ground.  
	We thank Roman Bezrukavnikov, Elijah Bodish, Pablo Boixeda-Alvarez, Kevin Coulembier, Pavel Etingof,  Jonathan Gruber, Mikko Korhonen, and Vera Serganova for stimulating conversations.  
	We further thank an anonymous referee for a particularly thorough and careful reading of our paper.  
	We would also like to thank Ruby, Heinrich and Lukas from the University of Sydney's One Tree Island Research Station for a stimulating research environment and a suitable name for our functor. 
	J.B. was supported by ARC grant DP220102861.  
	A.R. and T.F. were supported by a UNSW FRTG2024 grant.  
	A.S. was supported by ARC grant DP210100251 and by an AMS-Simons Travel Grant. 
	G.W. was supported by Australian Laureate Fellowship FL230100256 and the Max Planck Humboldt Research Award.

\section{Notation and background}
\label{sec: notation and background}

\subsection{Background on algebraic groups}
\label{Sec: Background alg grps}

    We begin by recalling basic facts and fixing notation pertaining to algebraic groups. 
    Standard references for this material include \cite{JantzenBook,WTakagi}.
    \par 
    Throughout we fix $\Bbbk$, an algebraically closed field of characteristic $p>0$.
    Later we will impose minor conditions on the characteristic $p$.

\subsubsection{Algebraic groups}
\label{sec: Alg grps}

    In this small section, we fix notation on subgroups and root datum, to be used throughout the paper. 

    We fix a reductive algebraic group, over $\Bbbk$, whose derived subgroup is simply connected.
    Further, we assume that our group $G$ arises via extension of scalars from a group over $\mathbb{F}_p$, and in particular have a fixed isomorphism $G \to G^{(1)}$ where $G^{(1)}$ denotes the Frobenius twist on $G$.
    Moreover, we fix a Borel subgroup $B \subset G$, and a maximal torus $T \subset B$. 
    We denote by $U \subset B$ the unipotent radical of $B$, and $h$ the Coxeter number of $G$. 
    \par 
    The root datum $(\mathfrak{X},R, \mathfrak{X}^{\vee}, R^{\vee})$ associated to $G$ consists of a character lattice $\mathfrak{X}$, root system $R \subset \mathfrak{X}$, cocharacter lattice $\mathfrak{X}^{\vee}$ and coroot system $R^{\vee} \subset \mathfrak{X}^{\vee}$.
    We fix a set of positive roots $R_+ \subset R$ and positive coroots $R^{\vee}_+ \subset R^{\vee}$
    so that  the roots occurring in the Lie algebra of $B$ are $- R_+$.
    The set of dominant weights will be denoted $\mathfrak{X}_+$.
    \par 
    We impose the following assumptions: 
    (a) The characteristic satisfies $p > h$; and 
    (b) The group $G$ is an almost-simple algebraic group
    (i.e. $Z(G)$ is finite and $G/Z(G)$ is simple).
    Assumption (a) \emph{is} necessary for various arguments in the paper. 
    Assumption (b) \emph{is not} necessary; it is made to simplify
    notation and exposition.

\subsubsection{Weyl groups and alcoves} 
\label{sec:Weyl groups and alcoves}

    Set $\mathfrak{X}_{\Rbb} := \mathfrak{X} \otimes_{\Zbb} \Rbb$. 
    For any $\alpha \in R$ we define the reflection $s_{\alpha} : \mathfrak{X}_{\Rbb} \rightarrow \mathfrak{X}_{\Rbb}$ by
    \begin{align}
    \label{Eqn: reflection action}
        s_{\alpha}(\lambda) = \lambda - \langle \lambda, \alpha^{\vee} \rangle \alpha
    \end{align}
    where $\lambda \in \mathfrak{X}_{\Rbb}$. 
    Let $\Sigma$ denote the set of simple roots in $R$, and $S_f$ the set of simple reflections $\{ s_{\alpha} ~|~ \alpha \in \Sigma \}$.
    The (finite) Weyl group $W_f = N_G(T)/T$ of $G$ is isomorphic to the group generated by the reflections $s \in S_f$.  
    The affine Weyl group\footnote{In the language of \cite{Bourbaki}, this would be called the affine Weyl group of ${}^L G$, the Langlands dual group of $G$. } $W$ and extended affine Weyl group $W^{\ext}$ are respectively defined as:
    \begin{align*}
        W = W_{f} \ltimes \Zbb R,
        &&
        W^{\ext} = W_{f} \ltimes \mathfrak{X},
    \end{align*}  
    where $\Zbb R$ denotes the root lattice of $G$.
    \par 
    We will use the $p$-dilated dot action of these groups on $\mathfrak{X}_{\Rbb}$. 
    For any $\mu \in \mathfrak{X}$ let $t_{\mu}$ denote $(1, \mu)$ in $W^{\ext}$. 
    Recall $\rho = \frac{1}{2} \sum_{\alpha \in R_+} \alpha$. 
    Then, for any $w \in W_f$, $\mu \in \mathfrak{X}$ and $\lambda \in \mathfrak{X}_{\Rbb}$, we define $w t_{\mu} \cdot \lambda$ as 
    \begin{align}
    \label{Eqn: rho shifted action}
        wt_{\mu} \cdot \lambda = w(\lambda + \rho + p\mu) - \rho.
    \end{align}
    Denote by $\alpha_0$ the highest short-root of $R$, and $s_0$ the  reflection $t_{\alpha_0} s_{\alpha_0}$.
    We call $S = S_f \cup \{ s_0 \}$ the set of simple reflections.
    The affine Weyl group is generated by the reflections $s \in S$.
    \par 
    Both $(W_f, S_f)$ and $(W,S)$ are Coxeter systems; the former being a standard parabolic subgroup of the latter. 
    Their Bruhat orders and length functions are denoted by $\leq$ and
    $\ell$ respectively.
    The group $W^{\ext}$ is not a Coxeter group in general, but still admits a length function, which we also denote by $\ell$.
    We denote the set of minimal length coset representatives for $W_f \backslash W$ by ${}^f W$. 
    \par 
    Consider the $p$-dilated, $\rho$-shifted fundamental alcove $A_0$ and its closure $\overline{A_0}$, which are defined as
    \begin{align*}
    	A_0 
    	&:= 
    	\{ \lambda \in \mathfrak{X}_{\Rbb} ~|~ 0 < \langle \lambda + \rho , \alpha^{\vee} \rangle  < p \text{ for all } \alpha \in R_+\},
    	\\
    	 \overline{A_0} 
    	 &:= 
    	 \{ \lambda \in \mathfrak{X}_{\Rbb} ~|~ 0 \leq \langle \lambda + \rho , \alpha^{\vee} \rangle  \leq p \text{ for all } \alpha \in R_+\}.
    \end{align*}  
    A connected component of $W \cdot A_0$ is called an alcove; the set of all alcoves is denoted $\mathscr{A}$.
    Any alcove $A \in \mathscr{A}$ that non-trivially intersects the set of dominant weights, i.e. $A \cap \mathfrak{X}_+ \neq \emptyset$, is called dominant, and the set of dominant alcoves is denoted $\mathscr{A}_+$.
    \par 
    The closure $\overline{A_0}$ is a fundamental domain for the $p$-dilated dot action of $W$ on $\mathfrak{X}_{\Rbb}$. 
    Consequently, we have a bijection $W \tilde{\longrightarrow} \mathscr{A}$ where $x \in W$ is identified with the alcove $x \cdot A_0$. 
    Moreover, this bijection restricts to a bijection ${}^f W \tilde{\longrightarrow} \mathscr{A}_+$.
	The extended affine Weyl group $W^{\ext}$ acts on $\mathscr{A}$; however, this action is not free. 
	The elements of $W^{\ext}$ which stabilise $A_0$ are denoted $\Omega$, and are  precisely the set of elements of length zero.
        
	Any element of $W^{\ext}$ may be regarded as an affine transformation of $\mathfrak{X}_{\Rbb}$ and hence as a linear transformation of $\mathfrak{X}_{\Rbb} \oplus \mathbb{R}$ in a standard way.
	Taking determinants gives rise to the sign character
	\[
		\epsilon : W^{\ext} \to \{ \pm 1 \}.
    \]
    This character restricts to the sign character $x \mapsto (-1)^{\ell(x)}$ on $W$. 
    However, $\epsilon$ is \emph{not} always equal to $y \mapsto (-1)^{\ell(y)}$ on $W^{\ext}$. 
    Rather, if we write $y = t_\mu x$, where $t_\mu$ is a translation and $x \in W$, then $\epsilon (y) = \epsilon (x)$, because translations always have determinant 1.
    
	Finally, a weight $\lambda \in W\cdot (\mathfrak{X} \cap A_0)$ is called $p$-regular; an integral weight that is not $p$-regular is called $p$-singular. 
    Equivalently, a weight $\lambda \in \mathfrak{X}$ is called $p$-regular if its stabiliser under the dot action of $W$ is trivial. 
    A $p$-regular weight exists if and only if $p\geq h$ by \cite[Equation (10) of \S 6.2]{JantzenBook}.

\subsubsection{Representations of algebraic groups}
\label{sec: Reps}

	Let $K$ be an algebraic group defined over $\Bbbk$.  
	The category of finite dimensional rational (equivalently algebraic) $\Bbbk$-representations of $K$ is denoted $\Rep (K)$.  
	More generally, if $\Cc$ is a symmetric tensor category {over $\Bbbk$} in the sense of \cite{EGNO}, then we may view $K$ as an algebraic group in $\Cc$ via the inclusion $\operatorname{Vec}_{\Bbbk}\sub\Cc$, where $\Vec_{\Bbbk}$ denotes the category of finite-dimensional vector spaces over $\Bbbk$. 
	We write $\Rep_{\Cc}(K)$ for the category of $K$-modules in $\Cc$.  
	More explicitly, objects of $\Rep_{\Cc}(K)$ are exactly objects in $\Cc$ with the structure of a right comodule over the coordinate algebra $\Bbbk[K]$.
	In particular, $\Rep (K)$ is, by definition, equal to $\Rep_{\Vec_{\Bbbk}} (K)$. 
	Later, we shall consider categories of the form $\operatorname{gr}_{{A}} \Rep_{\Cc}(K)$ for an abelian group $A$.
	Explicitly, this means that $\Bbbk[K]$ is an ${A}$-graded Hopf algebra, and $\operatorname{gr}_{{A}}\Rep_{\Cc}(K)$ is the category of ${A}$-graded $K$-modules in $\Cc$.    
	Equivalently, $K$ admits an action of a diagonalisable algebraic group $\mathbb{A}$ by automorphisms with $A=\mathfrak{X}(\mathbb{A})$, and we have $\operatorname{gr}_{A}\Rep_{\Cc}(K) \simeq\Rep_{\Cc}(\mathbb{A} \ltimes K)$.  
	(Here we use the notation $\mathfrak{X}(\Abb)$ for the group of characters of $\Abb$.)  
	We will primarily be concerned with the case $A=\Zbb$ and $\mathbb{A}=\Gbb_m$.
	\par 
	For each dominant, integral weight $\lambda \in \mathfrak{X}_+$ we consider the simple module $L_{\lambda}$, Weyl (standard) module $\Delta_{\lambda}$, induced (costandard) module $\nabla_{\lambda}$ and tilting module $T_{\lambda}$, each of highest weight $\lambda$. 
	Every simple, standard, costandard, and  indecomposable tilting module in $\Rep (G)$ is of the preceding form.

\subsubsection{Linkage classes and translation functors} 
	We refer to Part II, Chapter 7 of \cite{JantzenBook} for the following section.	
	For any $\lambda \in {\overline{A_0} \cap \mathfrak{X}}$, we define $\Rep_{\lambda} (G)$ and $\Rep_{\lambda}^{\ext} (G)$  to be the Serre subcategories of $\Rep (G)$ defined by  
	\begin{align*}
		\Rep_{\lambda} (G) 
		&=
		\langle L_{\mu}  ~|~ \mu \in  (W \cdot \lambda) \cap \mathfrak{X}_{+} \rangle
		\\
		\Rep_{\lambda}^{\ext} (G) 
		&=
		\langle L_{\mu}  ~|~ \mu \in (W^{\ext}\cdot \lambda) \cap \mathfrak{X}_{+} \rangle
	\end{align*}
	respectively.
	In general, $\Rep_{\lambda}(G)$ is not a block of $\Rep (G)$; it is a union of blocks. 
	However if $\lambda$ is $p$-regular, then $\Rep_{\lambda} (G)$ is a genuine block of $G$.  
	Our assumptions on $p$ ensure that $\lambda = 0$ is a $p$-regular weight. 
	In particular, we call $\Rep_{0} (G) $ the principal block of $\Rep (G)$, and $\Rep_{0}^{\ext} (G) $ the extended principal block of $\Rep (G)$ (though the latter is not a block). 
    The full subcategories of tilting modules in $\Rep (G)$, $\Rep_0 (G)$, and $\Rep_0^{\ext} (G)$  are respectively denoted by $\Tilt $, $\Tilt_0 $, and $\Tilt_0^{\ext} $.
	\par 
	Fix weights $\lambda, \mu \in \overline{A_0} \cap \mathfrak{X}$, and a representation $M$ with extremal weights contained in the orbit $W_f ({\lambda-\mu})$, where we use the standard action of $W_f$ on $\mathfrak{X}_{\Rbb}$.
	Let inc${}_{\lambda} : \Rep_{\lambda} (G) \rightarrow \Rep (G) $ denote the inclusion functor and proj${}_{\lambda} : \Rep (G) \rightarrow \Rep_{\lambda} (G) $ denote the projection functor. 
	Then we define the translation functor $\theta_{\mu}^{\lambda}$ as 
	\begin{align*}
		\theta_{\mu}^{\lambda} : \Rep_{\mu} (G) \longrightarrow \Rep_{\lambda} (G),
		&&
		V \longmapsto \text{proj}_{\lambda}(\text{inc}_{\mu}(V) \otimes M).
	\end{align*}
	Note that different choices of $M$ produce isomorphic functors $\theta^{\lambda}_{\mu}$. 
	Moreover, the functors $\theta_{\mu}^{\lambda}$ and $\theta^{\mu}_{\lambda}$ are biadjoint and exact. 
	\par 
	Again, fix a weight $\lambda \in A_0 \cap \mathfrak{X}$, a simple reflection $s \in S$, and a weight $\mu_s \in \overline{A_0} \cap \mathfrak{X}$ whose stabiliser under the $p$-dilated dot action of $W$ is exactly $\{ 1 , s\}$. 
	By our assumptions on $p$, such a weight $\mu_s$ exists. 
	The wall-crossing functor $\Theta_s$ is defined as the composition
	\begin{align*}
		\Theta_s 
		= 
		\theta^{\lambda}_{\mu_s} \circ \theta_{\lambda}^{\mu_s} :  \Rep_{\lambda} (G) \longrightarrow \Rep_{\lambda} (G).
	\end{align*}
	Again, $\Theta_s$ is only defined up to isomorphism. 
	The action of wall-crossing functors on standard and costandard objects is well-understood. 
	In particular, in the notation above, for $x$ and $xs$ in ${}^fW$ we have exact sequences
    \begin{align*}
        xs > x:
        &&
        0 \to \nabla_{x \cdot \lambda}~ \to  \Theta_{s}(\nabla_{x \cdot \lambda}) \to \nabla_{xs \cdot \lambda} \to 0, 
        \\
        xs < x:
        &&
        0 \to \nabla_{xs \cdot \lambda} \to  \Theta_{s}(\nabla_{x \cdot  \lambda}) \to \nabla_{x \cdot \lambda}~ \to 0.
    \end{align*}

\subsection{Background on $\Ver_p$ and the OTI functor}
\label{Sec: Background OTI}

\subsubsection{Stable module category and semisimplification}
\label{sec: stable module category and semisimplification}

	For an algebraic group ${F}$, we say that  $\Rep ({F})$ is Frobenius if it has enough projectives and the class of projective and injective modules coincide in $\Rep ({F})$. 
	For example, if $H$ is a finite group scheme and $K$ is a diagonalisable group which acts on $H$ by group automorphisms, then $\Rep(K\ltimes H)$ is Frobenius (which can be shown e.g.~by using the exact induction functor $\Ind_{K}^{K\ltimes H}(-):\Rep(K)\to\Rep(K\ltimes H)$.)

	\begin{definition}
 	Suppose that $\Rep ({F})$ is Frobenius.  
 	Define the stable module category $\Rep ({F})^{st}$ to be the quotient category by the ideal of morphisms that factor through a projective object.
	\end{definition}

	It is well known that the stable module category is tensor triangulated, with a symmetric monoidal structure inherited from $\Rep(F)$ (see \cite{MR1393196}). 
	The distinguished triangles of $\Rep({F})^{st}$ are those isomorphic to a rotation of triangles of the form 
	\[
		X\xrightarrow{a} Y\xrightarrow{b}Z\xrightarrow{c} X[1],
	\]
	where 
	\[
		0\to  X\xrightarrow{a} Y\xrightarrow{b}Z\to 0
	\]
	is a short exact sequence in $\Rep ({F})$.  
	Here, $X[1]:=I/X$, where $X\hookrightarrow I$ is an embedding of $X$ into an injective module $I$.  
	(It follows that $M[-1]:=\operatorname{ker}(P\twoheadrightarrow M)$, where $P$ is a projective module with a surjection onto $M$.) 
	The morphism $c:Z\to I/X$ is obtained by the following diagram, using the injectivity of $I$:
	\[
		\xymatrix{
		0 \ar[r] & X\ar[r]\ar@{=}[d] & Y\ar[d]\ar[r] & Z\ar[r] \ar@{-->}[d]^c & 0\\
		0 \ar[r] & X\ar[r] & I\ar[r] & I/X \ar[r] & 0.\\
		}
	\]

	For the following discussion we refer to \cite[Exercise 8.18.9]{EGNO}.  
	Recall that a morphism $f:X \rightarrow Y$ in $\Rep(F)$ is called negligible if $\mathrm{Tr}(gf)=0$ for all morphisms $g:Y \rightarrow X$.  
	The collection of negligible morphisms forms a tensor ideal in $\Rep(F)$.  
	We say an object $M$ is negligible if $\id_M$ is negligible, or, equivalently, if $M$ is a direct sum of indecomposable objects of dimension divisible by $p$.
	
	\begin{definition}
    Let ${F}$ be an algebraic group. 
 	Define the semisimplification of $\Rep({F})$, written $\Rep({F})^{ss}$, to be the quotient of $\Rep({F})$ by the ideal of negligible morphisms. 
 	We call the quotient functor $\Rep(F)\to\Rep(F)^{ss}$ the semisimplification functor.  
	\end{definition}
	
	Note that $\Rep(F)^{ss}$ is a semisimple category, and has simple objects given by indecomposable $F$-modules $M$ such that $(\dim(M),p)=1$.

	Both the semisimplification $\Rep({F})^{ss}$ and the stable module category $\Rep({F})^{st}$ (when defined) admit natural quotient functors from $\Rep({F})$. 
	By \cite[Exercise 8.18.9]{EGNO} the ideal of morphisms factoring through a projective object is contained in the maximal ideal of negligible morphisms, so we have a factorisation:
	\[
		\xymatrix{
		\Rep({F})\ar[r]\ar[rd] & \Rep({F})^{st} \ar[d] \\ 
		& \Rep({F})^{ss}.
		}
	\]

\subsubsection{Representations of $C_p$ and $\alpha_p$}
\label{sec: Reps Cp alphap}

	Write $C_p$ for the finite cyclic group of order $p$, and let $\sigma\in C_p$ be a chosen generator.  
	Set $N:=1-\sigma\in\Bbbk C_p$, so that we have a presentation $\Bbbk C_p\cong\Bbbk[N]/N^p$.  

	In parallel, write $\alpha_p$ for the finite additive group scheme $\Spec\Bbbk[x]/x^p$.  
	The distribution algebra of $\alpha_p$ is naturally presented as $\Bbbk[E]/E^p$, where $E$ is primitive.  

	It follows that $\Rep(\alpha_p)\simeq\Rep_{\Bbbk}(C_p)$ as abelian categories, and this equivalence is compatible with their fibre functors to $\operatorname{Vec}_{\Bbbk}$.  
	Each has $p$ indecomposables, which we denote by $M_0,\dots,M_{p-1}$ where $\dim M_{i}=i+1$ (we abuse notation and write $M_i$ for objects in each category).  
	Note that $M_{p-1}$ is projective, and, in fact, is the free module of rank 1.

	It is clear that both $\Rep (C_p)$ and $\Rep (\alpha_p)$ are Frobenius, and that the indecomposable objects of the stable category and the semisimplification are given by the images of $M_0,\dots,M_{p-2}$.  
	Thus each category `remembers' the isomorphism class of an object up to projective summands.

	\begin{definition}
		Define the Verlinde-$p$ category by $\operatorname{Ver}_p:=\Rep (\alpha_p)^{ss}$.  
	\end{definition}

	\begin{remark}
    	Both $\Rep(\alpha_p)^{st}$ and $\Ver_p$ have $p-1$ indecomposable objects given by the images of $M_0,\dots,M_{p-2}$ under the respective quotients $\Rep(\alpha_p)\to\Rep(\alpha_p)^{st}$ and $\Rep(\alpha_p)\to\Ver_p$.  
    	Further, they are both monoidal categories, and their tensor product rules are the same in terms of indecomposable objects.  
    	However, $\Rep(\alpha_p)^{st}$ is not an abelian category, as, for instance, there are nontrivial triangles $M_i\to M_{i+j+1}\to M_j\to M_i[1]$ for $0\leq i,j,i+j+1\leq p-2$.
	\end{remark}

	We will write $L_0,\dots,L_{p-2}$ for the (isomorphism classes of) simple objects of $\operatorname{Ver}_p$, where $L_i$ is the image of the object $M_i$ under the semisimplification functor.  
	We have the well-known tensor product formula (see, for instance, \cite[\S 2.3]{EOfrob})
	\[
		L_{i-1}\otimes L_{j-1}\cong\bigoplus\limits_{k=1}^{\min(i,j,p-i,p-j)}L_{|i-j|+2k-2}.
	\]
	In particular, $L_{p-2}^{\otimes2}\cong L_0$, so $L_0$ and $L_{p-2}$ generate a tensor subcategory of $\operatorname{Ver}_p$.  
	If $p>2$, it is well known that this tensor subcategory is $\operatorname{sVec}$, where $L_{p-2}$ corresponds to an odd, one-dimensional super vector space.  
	We will write $\Pi$ for the endofunctor of $\operatorname{Ver}_p$ given by $L_{p-2}\otimes(-)$.  
	The following lemma is left as an exercise.

	\begin{lemma}
	\label{lemma shift Pi}
		Write $Q:\Rep(\alpha_p)^{st}\to\Rep(\alpha_p)^{ss}=\operatorname{Ver}_p$ for the quotient functor from the stable category to the semisimplification.  
    	Then we have an isomorphism of functors:
    	\[
    		Q\circ[1]\simeq \Pi\circ Q.
		\]
	\end{lemma}

	In the following lemma, choose a root subgroup $\Gbb_a\sub SL_2$, and consider the subgroups $C_p=\Gbb_a(\mathbb{F}_p)\sub SL_2$ and $\alpha_p=(\mathbb{G}_a)_1\sub SL_2$, where we use the notation $(-)_1$ for the first Frobenius kernel.  

	Let $\operatorname{Tilt}(SL_2)$ denote the category of tilting modules for $SL_2$, which is a pseudo-tensor category, meaning it is $\Bbbk$-linear, Karoubian, symmetric, monoidal, and rigid. 
	One may still talk about the tensor ideal of negligible morphisms inside of $\operatorname{Tilt}(SL_2)$, using the same definition as in $\S \ref{sec: stable module category and semisimplification}$. 
	We say a tilting module $T$ is negligible if $\id_T$ is negligible.  
	By definition, the negligible tilting modules are exactly direct sums of indecomposable tilting modules of dimension divisible by $p$.

	\begin{lemma}
	\label{lemma tiltings sl2}  
		For a tilting module $T$ of $\operatorname{SL}_2$, the following are equivalent:
    	\begin{enumerate}
        	\item $T$ is negligible,
        	\item $T$ is a direct sum of indecomposable tilting modules $T_i$ for $i\geq p-1$,
        	\item $T|_{\alpha_p}$ is projective, and
        	\item $T|_{C_p}$ is projective.
    	\end{enumerate}
	\end{lemma}

	\begin{proof}
    	The Steinberg module $T_{p - 1}$ satisfies the above four statements: it is negligible, and by inspection it is projective over $\alpha_p$ and $C_p$. 
    	The tilting modules in the thick tensor ideal generated by $T_{p-1}$ are therefore also both negligible and projective over $\alpha_p$ and $C_p$. 
    	But this includes all tilting modules $T_i$ for $i\geq p-1$.

    	On the other hand, $\dim(T_i) = i + 1$ for $i < p - 1$, so in particular $T_i$ is not negligible for $i<p-1$, and is, in particular, not projective over $\alpha_p$ or $C_p$.  
    	It follows that the negligible tilting modules are exactly direct sums of $T_i$ for $i\geq p-1$, and our statement follows.
	\end{proof}

	\begin{cor}
	\label{cor negl obj negl more}
 		In $\operatorname{Tilt}(SL_2)$, a morphism $T\to T'$ is negligible if and only if it factors through a negligible tilting module.   
	\end{cor}
	\begin{proof}
    	The backwards direction is clear.  
    	For the forwards direction, it suffices to show the statement when $T$ and $T'$ have no negligible summands. 
    	But this means they are direct sums of the tilting modules $T_0,\dots,T_{p-2}$.
		Since these tilting modules are irreducible, the statement immediately follows.
	\end{proof}

	\begin{lemma}
    	There is a symmetric monoidal equivalence $\Rep (C_p)^{ss}\simeq\operatorname{Ver}_p$.
	\end{lemma}
	\begin{proof}
		One may define the semisimplification of $\operatorname{Tilt}(SL_2)$ to be the quotient by the tensor ideal of negligible morphisms, and it will be a semisimple tensor category.

		We have embeddings $C_p,\alpha_p \sub SL_2$ which are both unique up to conjugacy.
		Using the explicit description of negligible tilting objects in Lemma \ref{lemma tiltings sl2} along with Corollary \ref{cor negl obj negl more}, we see that if a morphism in $\operatorname{Tilt}(SL_2)$ is negligible then its restriction to either $C_p$ or $\alpha_p$ is also negligible.  
		Thus we obtain the following diagram:
		\[
		\xymatrix{
 			& \operatorname{Tilt}(SL_2) \ar[rd]^{Res}\ar[dl]_{Res} \ar[dd]&\\ 
 			\Rep (C_p) \ar[dd] & & \Rep (\alpha_p) \ar[dd] \\
 			& \operatorname{Tilt}(SL_2)^{ss}\ar[rd]^{R_1}\ar[dl]_{R_2} & \\
 			\Rep (C_p)^{ss} & & \Rep (\alpha_p)^{ss}.
		}
		\]
		It is easy to see that $R_1,R_2$ are symmetric monoidal equivalences.
		Thus we obtain our desired equivalence as $R_2\circ R_1^{-1}$.
	\end{proof}

	We will from now on identify $\Rep (C_p)^{ss}$ with $\operatorname{Ver}_p$.

\subsubsection{Conjugacy classes and the nilpotent cone} 
\label{sec: Reductive gps}

	We continue with the notation established in \S\ref{sec: Alg grps}.  
	Write $\Nc\sub\operatorname{Lie}G$ for the nilpotent cone of $G$. 
	Then we have bijections  (see \cite{springerunipvar}, {\cite[\S 6.20]{MR1343976}})
	\begin{equation}
    \label{eq: bijections}
		\{C_p \subseteq G\}/G
		\longleftrightarrow 
		\{\alpha_p\subseteq G\}/G 
		\longleftrightarrow 
		\Nc/G. 
	\end{equation}

	Recall that $\Nc$ admits a dense open orbit under $G$, known as the regular orbit.  

	\begin{definition}
	\label{def: regular alpha p}
    	We say that a subgroup $H\sub G$, where $H\cong C_p$ or $H\cong\alpha_p$, is regular if it corresponds to a regular orbit under the bijections in \eqref{eq: bijections}.
	\end{definition}

	\begin{lemma}
	\label{lemma centraliser description}
    	Let $H$ be a regular subgroup lying in $B$.  
    	Writing $C_G(H)$ for the centraliser subgroup of $H$ in $G$, we have 
    	\[
    		C_G(H)\cong\Gbb_a^{\operatorname{rk}(G)}\times Z(G).
    	\]
	\end{lemma}

	\begin{proof}
    	By \cite{springer}, $C_G(H)\cong C_U(H)\times Z(G)$ where $U$ is the unipotent radical of $B$, and $C_U(H)$ is smooth and commutative of dimension $\operatorname{rk}(G)$.  
    	Since it is also unipotent, we deduce that $C_U(H)\cong \Gbb_a^{\operatorname{rk}(G)}$.
    \end{proof}

\subsubsection{The principal $SL_2$ subgroup}
\label{sec prin SL2} 
	Suppose that $H\sub G$ is regular.  
	Then there exists a principal $SL_2$-subgroup $K\sub G$,  for which $H\sub K$. 
	In fact, we may choose $H$ to lie in a root subgroup of $K$.

	Given $T\sub B\sub G$, we can, and will, always choose a principal $SL_2$-subgroup $K\sub G$ such that $T_K:=T\cap K$ is a maximal torus of $K$, and $B_K:=B\cap K$ is a Borel subgroup of $K$.  
	We note that in this case, $T_K\cong\Gbb_m\sub T$ is given by the coweight $2\rho^\vee\in\mathfrak{X}^\vee$.

\subsubsection{The OTI functor}
\label{sec: OTI}

	Choose a regular subgroup $H\sub G$ with either $H\cong C_p$ or $H\cong\alpha_p$. Write $N_G(H)$ for the normaliser subgroup of $H$ in $G$, and $\cf=\operatorname{Lie}C_G(H)$.  
	Write $\mathbb{A}_H=N_G(H)/C_G(H)$.  
	Then we have a splitting $N_G(H)\cong\mathbb{A}_H\ltimes C_G(H)$, and $\mathbb{A}_H\cong\Gbb_m$ if $H\cong\alpha_p$, and $\mathbb{A}_H\cong\Fbb_{p}^{\times}$ if $H\cong C_p$.  
	Finally, set $A_H=\mathfrak{X}(\mathbb{A}_H)$, so that $A_H\cong\Zbb$ if $H\cong\alpha_p$ and $A_H\cong\Zbb_{p-1}$ if $H\cong C_p$.

	We define an OTI functor $\Phi_H:\Rep (G)\to\operatorname{Ver}_p$ to be one given by the following composition:
	\[
	\xymatrix{
		\Rep (G)\ar[r]^{Res} \ar[dr]_{\Phi_H} & \Rep (H)\ar[d]^{ss}\\
 		& ~\operatorname{Ver}_p.
	}
	\]
	This functor has the following explicit description, as explained in \cite[\S 4]{EOfrob}.  
	We may write
	\[
		\Phi_H(V)=\bigoplus\limits_{i=0}^{p-2}\Phi^i_H(V)\otimes L_i,
	\]
	where $\Phi_H^i$ is the functor given by
	\begin{equation}
	\label{eqn oti explicit formula}
    	\Phi^i_H(V)
    	:=
    	\frac{\ker(\eta)\cap\operatorname{im}(\eta^i)}{\ker(\eta)\cap\operatorname{im}(\eta^{i+1})},
	\end{equation}
	where $\eta=E$ if $H\cong \alpha_p$ and $\eta=N$ if $H\cong C_p$.  
	Our presentation shows that $N_{G}(H)$ naturally acts on $\Phi_H$.  
	On the other hand, the action of $\gf$ on modules in $\Rep(G)$ induces an action of $\Phi_H(\gf)$ on $\Phi_H$, where $\Phi_H(\gf)$ is a Lie algebra in $\Ver_p$.  
	By the table in 3.2.4 of \cite{coulembier2025finite}, $\Phi_H(\gf)$ (1) has no summands isomorphic to $L_0$, and (2) is multiplicity free as an object of $\Ver_p$.  
	Point (1) implies that $\Phi_H(\cf)=\cf\to\Phi_H(\gf)$ is the zero map, which means that $\Phi_H(\cf)$, and thus $C_G(H)_1$, acts trivially on $\Phi_H$.  
	It is clear that $Z(G)$ acts trivially on $\Phi_H(\gf)$, and since $C_G(H)/Z(G)$ is unipotent (Lemma \ref{lemma centraliser description}), point (2) implies that $C_G(H)$ acts trivially on $\Phi_H(\gf)$.  
	Putting this together, we may view $\Phi_H$ as a functor
	\[
		\Phi_H:\Rep (G)\to\operatorname{gr}_{A_H}\Rep_{\operatorname{Ver}_p}(C_{G}(H)^{(1)}\times\Phi_H(\gf)).
	\]
	Note that we write $\Rep_{\operatorname{Ver}_p}(C_{G}(H)^{(1)}\times\Phi_H(\gf))$ for the category of objects in $\Ver_p$ admitting commuting actions of the algebraic group $C_G(H)^{(1)}$ and the Lie algebra $\Phi_H(\gf)$.

	We will refer to the functors $\Phi_H$ constructed above as OTI functors.  
	We note that they are symmetric monoidal but not exact.  
	\begin{remark}
		Note that if we choose two conjugate subgroups $H,H'\sub G$, we will obtain isomorphic functors $\Phi_H\cong\Phi_{H'}$.  
		Thus the only meaningful choice is whether $H\cong C_p$ or $H\cong\alpha_p$.
	\end{remark}

	\begin{remark}
    	It would be interesting to study the similarly defined functors $\Phi_H$ when $H\sub G$ is not regular.
	\end{remark}
	\begin{lemma}
	\label{OTI 0 iff proj and SES}
    	For $V\in\Rep(G)$, $\Phi_H(V)=0$ if and only if $V|_{H}$ is projective.  
    	In particular, for a short exact sequence
   		\[
    		0\to X\to Y\to Z\to 0,
    	\]
    	if any two of $\Phi_H(X),\Phi_H(Y)$, or $\Phi_H(Z)$ are 0, then so is the third.
	\end{lemma}
	\begin{proof}
    The first statement is because an indecomposable $M\in\Rep(H)$ is negligible if and only if $M$ is projective.  
    The second statement follows from the fact that $\Rep(H)$ is Frobenius.
	\end{proof}

	\begin{cor}
	\label{tilt 0 implies all 0}
    	Suppose that $G$ is reductive and $\mathcal{B}\sub\Rep(G)$ is a block such that $\Phi_H(T)=0$ for all tilting modules $T$ in $\mathcal{B}$.  
    	Then $\Phi_H(V)=0$ for all $V\in\mathcal{B}$.
	\end{cor}
	\begin{proof}
    	Indeed, for every module $V\in\mathcal{B}$ there exists a bounded complex $T^\bullet$ of tilting modules such that $H^0(T^\bullet)=V$ and $H^i(T^\bullet)=0$ for $i\neq0$ (we use, e.g.,~the equivalence $K^b(\operatorname{Tilt})\simeq D^b(\Rep G)$, see \S\ref{Sec: complexes of tilting modules}). 
    	By repeatedly applying  Lemma \ref{OTI 0 iff proj and SES}, it is an exercise to show that $\Phi_H(V)=0$.
	\end{proof}

\subsection{Blocks and the OTI functor}
\subsubsection{Overview} 
	As the reader will have gleaned from the introduction, our results concerning the OTI functor are cleanest for the extended principal block. 
	However it is natural to ask what happens for other blocks. 
	To this end, consider the translation functor
	\[
		\theta_{\mu}^{\lambda} : \Rep_{\mu}(G) \to \Rep_{\lambda}(G).
	\]
One would like to fill in the dashed arrow below for a commutative diagram:
\[
    \xymatrix{
    	\Rep_{\mu}(G)\ar[rrr]^{\theta_{\mu}^{\lambda}} \ar[d]_{\Phi_H} & & & \Rep_{\lambda}(G)\ar[d]^{\Phi_H}
    	\\
    	\Ver_p \ar@{-->}[rrr]& & & \Ver_p
    }
    \]
    Translation functors are built out of tensor products and projections.  The former is easy to extend using the monoidality of $\Phi_H$, however more work is required to extend the projection.

	The point of this section is that if one uses more of the theory of tensor categories in $\Ver_p$,  then a bottom arrow can be constructed and made into an equivalence, and standard arguments involving translation functors carry through.
	The arguments of this section need more machinery from the theory of tensor categories, and are not essential to understand the main theorems of the paper.

\subsubsection{OTI behaves the same on regular blocks}
\label{section oti blocks}

	Write $\operatorname{gr}_{A_H}\Rep_{\Ver_p}(Z\times\Phi_H(\gf))$ for the category of $A_H$-graded\footnote{Recall $A_H\cong\Zbb$ if $H\cong\alpha_p$, and $A_H\cong\Zbb_{p-1}$ if $H\cong C_p$.} representations of $\Phi_H(\gf)$ in $\Ver_p$ with a commuting action of $Z=Z(G)$.  
	By \cite[Theorem 3.3.2]{coulembier2025finite} and \cite[Lemma 4.2.3]{MR4794815}, 
	\begin{equation}
	\label{eqn equiv rep cats}
		\operatorname{gr}_{A_H}\Rep_{\Ver_p}(Z\times\Phi_H(\gf))\simeq\operatorname{gr}_{A_H}(\Ver_p\boxtimes\Ver_p(G)),  
	\end{equation}
	where $\Ver_p(G)$ is the semisimplification of $\operatorname{Tilt}(G)$, and $\boxtimes$ denotes the Deligne tensor product of symmetric monoidal categories (see, e.g., \cite[\S4.6]{EGNO}).  
	In other words, this representation category is semisimple with simple representations given by the graded shifts of $L_i\boxtimes\Phi_H(L_{\lambda})$ for $\lambda\in A_0$ and $i=0,\dots,p-2$.  
	To be more explicit, $L_i$ is given the trivial grading, whereas $\Phi_H(L_{\lambda})$ has grading coming from the action of $\mathbb{A}_H$. 
	Further, $Z\times\Phi_H(\gf)$ acts trivially on the $L_i$ factor of $L_i\boxtimes\Phi_H(L_{\lambda})$.
	By Lemma \ref{lemma centraliser description}, we have
	\[
		C_G(H)^{(1)}\cong\Gbb_a^{\operatorname{rk}(G)}\times Z.
	\]
	Because $\Gbb_a^{\operatorname{rk}(G)}$ is unipotent, the restriction functor 
	\[
		\operatorname{gr}_{A_H}\Rep_{\Ver_p}(C_G(H)^{(1)}\times\Phi_H(\gf))
		\to
		\operatorname{gr}_{A_H}\Rep_{\Ver_p}(Z\times\Phi_H(\gf))
	\]
	takes simples to simples and induces a bijection on simple objects.

	\begin{definition}
    	For $\lambda\in A_0  \cap \mathfrak{X}$, define $\operatorname{gr}_{A_H}\Rep_{\lambda}(C_G(H)^{(1)}\times\Phi_H(\gf))$ to be the Serre subcategory of $\operatorname{gr}_{A_H}\Rep_{\Ver_p}(C_G(H)^{(1)}\times\Phi_H(\gf))$ generated by all graded shifts of the simples $L_i\boxtimes\Phi_H(L_\lambda)$ for $i=0,1,\dots,p-2$. 
	\end{definition} 
	 Because $C_G(H)^{(1)}$ commutes with $\Phi_H(\gf)$, there are no nontrivial extensions between the simples $L_i\boxtimes\Phi_H(L_{\lambda})$ and $L_j\boxtimes\Phi_H(L_{\mu})$ when $\lambda\neq \mu$. 
	Thus we have a decomposition
	\[
		\operatorname{gr}_{A_H}\Rep_{\Ver_p}(C_G(H)^{(1)} \times \Phi_H(\gf))
		=
		\bigoplus\limits_{\lambda\in A_0 \cap \mathfrak{X}}\operatorname{gr}_{A_H}\Rep_{\lambda}(C_G(H)^{(1)}\times\Phi_H(\gf)).
	\]
	Further, the functor  
	\begin{eqnarray*}
		\operatorname{gr}_{A_H}\Rep_{\Ver_p}(\Gbb_a^{\operatorname{rk}(G)})& \to & \operatorname{gr}_{A_H}\Rep_{\lambda}(C_G(H)^{(1)}\times\Phi_H(\gf)), 
		\\
		V& \mapsto &V\otimes\Phi_{H}(L_{\lambda}),
	\end{eqnarray*}
 	is an equivalence of abelian categories for all $\lambda\in A_0\cap \mathfrak{X}$.  
 	Indeed, an inverse functor is given by $\underline{\Hom}_{\Phi_H(\gf)}(\Phi_{H}(L_{\lambda}),-)$, which has an obvious action of $C_G(H)^{(1)}$ (here $\underline{\Hom}$ denotes the internal $\Hom$ in $\Ver_p$).  

	By our assumptions and \cite[Theorem 3.1]{MR283038}, the
        blocks $\Rep_{\lambda}(G)$ are separated by the joint actions
        of $Z$ and the Harish-Chandra part of the centre of the enveloping algebra of $\gf$.  
	Since the centre of the enveloping algebra acts on $\Phi_H$ and commutes with $N_G(H)\ltimes\Phi_H(\gf)$, we obtain that, for $\lambda\in A_0 \cap \mathfrak{X}$,
	\begin{equation}
	\label{eqn block to block}
		\Phi_H(\Rep_{\lambda}(G))\sub\operatorname{gr}_{A_H}\Rep_{\lambda}(C_G(H)^{(1)}\times\Phi_H(\gf)).    
	\end{equation}
	\begin{lemma}
	\label{lemma trans OTI}
    	For $\lambda,\mu\in A_0 \cap \mathfrak{X}$, the translation functors $\theta^{\lambda}_{\mu}$ induce equivalences 
    	\[
    		\ol{\theta}_{\mu}^{\lambda}:\operatorname{gr}_{A_H}\Rep_{\mu}(C_G(H)^{(1)}\ltimes\Phi_H(\gf))
    		\to
    		\operatorname{gr}_{A_H}\Rep_{\lambda}(C_G(H)^{(1)}\ltimes\Phi_H(\gf)),
    	\]
    	and the following diagram commutes:
    	\[
    	\xymatrix{
    		\Rep_{\mu}(G)\ar[rr]^{\theta_{\mu}^{\lambda}} \ar[d]_{\Phi_H} & & \Rep_{\lambda}(G)\ar[d]^{\Phi_H}
    		\\
    		\operatorname{gr}_{A_H}\Rep_{\mu}(C_G(H)^{(1)}\ltimes\Phi_H(\gf))\ar[rr]^{\ol{\theta}_{\mu}^{\lambda}} & & \operatorname{gr}_{A_H}\Rep_{\lambda}(C_G(H)^{(1)}\ltimes\Phi_H(\gf))
    	}
    	\]
	\end{lemma}
	\begin{proof}
    If $\theta_{\mu}^{\lambda}=\operatorname{pr}_{\lambda}(M\otimes \text{inc}_{\mu}(-))$, then we may set $\ol{\theta}^{\lambda}_{\mu}=\operatorname{pr}_{\lambda}(\Phi_{H}(M)\otimes (-))$, and the commutativity of the diagram is clear.  
    It is easy to check that $\ol{\theta}_{\lambda}^{\mu}$ is biadjoint to $\ol{\theta}_{\mu}^{\lambda}$.  
    By the commutativity of the square, we have 
    \[
    	\ol{\theta}^{\lambda}_{\mu}(L_i\boxtimes\Phi_H(L_\mu))
    	\cong 
    	L_i\boxtimes\ol{\theta}^{\lambda}_{\mu}(\Phi_H(L_{\mu}))
    	=
    	L_i\boxtimes \Phi_H(\theta_{\mu}^{\lambda}(L_\mu))
    	=
    	L_i\boxtimes \Phi_H(L_{\lambda}).
    \]
    Therefore, $\ol{\theta}^{\lambda}_{\mu}$ takes simples to simples.  
    One may now apply an analogous argument to \cite[II, Proposition 7.9]{JantzenBook} to show that $\ol{\theta}^{\lambda}_{\mu}\circ\ol{\theta}^{\mu}_{\lambda}$ and $\ol{\theta}^{\mu}_{\lambda}\circ\ol{\theta}^{\lambda}_{\mu}$ are isomorphic to the identity.
    \end{proof}

\section{Proof of Theorem A}
\label{Sec: Proof Thm A}

	We continue to use the notations established in \S\ref{sec: Alg grps}, \ref{sec prin SL2}. 
 	Let $H\sub G$ be either $H=U_K(\mathbb{F}_p)\cong C_p$ or $H = (U_K)_{1} \cong \alpha_p$.

\subsection{Costandard modules and $\Phi_H$}
\label{sec: Costandards}

	In this section we show part (1) of Theorem A, which says that $\Phi_H(\nabla_{\lambda})$ is determined, up to shift, by $\Phi_H(\nabla_{\lambda_0})$, where $\lambda_0 \in A_0$ is the unique weight in the closure of the fundamental alcove in the orbit of $\lambda$.  
	We make essential use of the following theorem of Jantzen. 

	\begin{thm}
	\label{Thm: Costandard p-singular}
     	If $\lambda$ is $p$-singular and $V\in\Rep_{\lambda}(G)$, then $\Phi_H(V)=0$.  
     	In particular, $\Phi_H(\nabla_{\lambda}) =0$.  
	\end{thm}

	\begin{proof}
		By Corollary \ref{tilt 0 implies all 0}, it suffices to show that $\Phi_H(T)=0$ for all tilting modules lying in $\Rep_{\lambda}(G)$. 
		Lemma \ref{OTI 0 iff proj and SES} implies $\Phi_H(T)=0$ if and only if $T|_H$ is projective.
		Then, by Lemma \ref{lemma tiltings sl2}, it suffices to consider only the case when $H \cong \alpha_p$.
 		Since every tilting module has a $\nabla$-filtration, it suffices to show that $\Phi_H(\nabla_{\mu})=0$ for any $p$-singular dominant weight $\mu$.  
 		By \cite[Corollary 6.8]{SFB} this is equivalent to the support variety of $\nabla_{\mu}$ being properly contained in the nilpotent cone $\Nc$, which is the case by \cite[Satz 4.14]{JantzenKohomologie}.
	\end{proof}

	We will provide another proof of Theorem \ref{Thm: Costandard p-singular} using coherent geometry (when $G$ is not of type $\textbf{G}_2,\textbf{F}_4,$ or $\textbf{E}_8$) in \S\ref{sec: vanishing p-sing geom}.

	\begin{cor}
	\label{cor oti tilts}
		For $\lambda\in\mathfrak{X}_+$, $\Phi_H(T_{\lambda})=0$ if and only if $\lambda\notin A_0$.
	\end{cor}
	\begin{proof}
		By Theorem \ref{Thm: Costandard p-singular}, $\Phi_H(T_{\lambda})=0$ if $\lambda$ is $p$-singular.  
		If $\lambda\in\mathfrak{X}_+$ is $p$-regular and $\lambda\notin A_0$, then a standard argument using translation functors shows that $T_{\lambda}$ is a direct summand of some $T_{\mu}\otimes T_{\nu}$, where $\mu$ is $p$-singular.  
		Because $\Phi_H$ is monoidal, we obtain that $\Phi_H(T_{\lambda})=0$.  

		If $\lambda\in \mathfrak{X}_+ \cap A_0$ then $T_{\lambda}=\nabla_{\lambda}$.  
		In this case, the Weyl dimension formula shows that $\dim\nabla_{\lambda}$ and $p$ are coprime, and in particular $\Phi_H(T_{\lambda})\neq0$.  
	\end{proof}

	\begin{cor} 
	\label{cor:shifting}
	Suppose $\lambda \in \mathfrak{X}_+$ lies in the closure of an alcove $A$, and $s$ is an affine reflection about a wall of $A$ such that $s\cdot\lambda\in\mathfrak{X}_+$. 
	Then we have an isomorphism of ${C_{G}(H)^{(1)}\times\Phi_H(\gf)}$-modules \linebreak $\Phi_H (\nabla_{s\cdot\lambda}) \cong \Pi \Phi_H ( \nabla_{\lambda})$. 
	\end{cor}

	Note that Corollary \ref{cor:shifting} implies (1) of Theorem \ref{thm A} since all reflections in $W$ have odd length.

	\begin{proof}
    	If $\lambda$ is $p$-singular, the claim is true by Theorem \ref{Thm: Costandard p-singular}. 
    	If $\lambda$ is regular, then there exists $\lambda_0 \in A_0$, and $x \in {}^f W$ such that $\lambda = x \cdot \lambda_0$.
    	For $s' \in S$ with $xs'\cdot \lambda \in \mathfrak{X}_+$, we define $\lambda'$ to be $xs'\cdot \lambda$.  
    	Let us assume without loss of generality that $xs' > x$, and consider the exact sequence 
    	\begin{equation}
    	\label{eqn wall crossing SES}
        	0 
        	\to 
        	\nabla_{\lambda} 
        	\to  
        	\Theta_{s'}(\nabla_{\lambda}) 
        	\to 
        	\nabla_{\lambda'} 
        	\to 
        	0.
    	\end{equation}
    	Recall $\Theta_{s'} = \theta_{\mu_{s'}}^{\lambda_0} \circ \theta_{\lambda_0}^{\mu_{s'}}$ for some $p$-singular weight $\mu_{s'}$ whose stabiliser is generated by $s'$. 
    	In particular, $\theta_{\lambda_0}^{\mu_{s'}}(\nabla_{\lambda})$ lies in a union of blocks $\Rep_{\mu}(G)$ for $p$-singular $\mu$, and thus $\Phi_H(\theta^{\mu_{s'}}_{\lambda_0}(\nabla_{\lambda}))=0$ by Theorem \ref{Thm: Costandard p-singular}. 
    	Passing to the stable category of $H$, (\ref{eqn wall crossing SES}) becomes an exact triangle where the middle term is 0 by Lemma \ref{OTI 0 iff proj and SES}, so we obtain
    	\begin{equation*}
        	\nabla_{\lambda'}
        	\cong 
        	\nabla_{\lambda}[1].
    	\end{equation*}
    	Further passing to the semisimplification $\Ver_p$ proves the claim, by Lemma \ref{lemma shift Pi}.
	\end{proof}

\subsection{Coherent geometry of $\alpha_p$ and $C_p$ actions}
\label{sec: Coherent geometry}

	The arguments of this section draw heavily from the ideas in \cite{sslocalization}.

 	Let $X$ be a separated, finite type scheme over $\Bbbk$, and write $\Oc:=\Oc_{X}$ for the sheaf of regular functions.  
 	If $K$ is an algebraic group, a left $K$-action on $X$ is the data of a morphism $a:K\times X\to X$ satisfying the usual axioms.  
 	We will occasionally abbreviate this data by saying that $X$ is a $K$-scheme. 
 	We leave to the reader the verification of the following:
 	\begin{itemize}
     \item The data of an $\alpha_p$-action on $X$ is equivalent to the data of a global vector field $E$ on $X$ with $E^p=0$.
     \item The data of a $C_p$-action on $X$ is equivalent to the data of an automorphism $\sigma$ of $X$ satisfying $\sigma^p=\id$.  
     In this case we write $\sigma^*$ for the isomorphism of regular functions $\sigma_*\Oc\xrightarrow{\sim}\Oc$, and set $u^*:=1-\sigma^*$.
 	\end{itemize}

	Suppose that $a:K\times X\to X$ is a $K$-action, and write $p_2:K\times X\to X$ for the projection.  
	A $K$-equivariant quasicoherent sheaf $\Fc$ on $X$ is the data of an isomorphism $a^*\Fc\cong p_2^*\Fc$ of quasicoherent sheaves satisfying the usual cocycle condition (see \cite[\S 5.1]{chrissginzburg}).  
	One may verify the following: 
	\begin{itemize}
    	\item An $\alpha_p$-equivariant quasicoherent sheaf $\Fc$ is the data of a sheaf endomorphism $E$ (by abuse of notation) of $\Fc$ such that for sections $f$ of $\Oc$ and $s$ of $\Fc$, we have
		\[
			E(fs)=E(f)s+fE(s).
		\]
    \item A $C_p$-equivariant quasicoherent sheaf $\Fc$ is the data of a sheaf isomorphism $\sigma:\sigma_*\Fc\to \Fc$ such that for a section $f$ of $\sigma_*\Oc$ and $s$ of $\sigma_*\Fc$, we have $\sigma(fs)=\sigma^*(f)\sigma(s)$.  In particular, setting $u:=1-\sigma$ we obtain that:
    	\[
   			u(fs)=u^*(f)s+\sigma^*(f)u(s).
    	\]
	\end{itemize}

	We say that an action of $K$ on $X$ is free if the morphism $a\times p_1:K\times X\to X\times X$ is a closed embedding.  
	By \cite[Chapter III, \S 2, 2.5]{demazuregabriel}, this is equivalent to asking that $K(\Bbbk)$ acts freely on $X(\Bbbk)$ and that the isotropy Lie subalgebra vanishes at every point.  
	If $Y$ is a $K$-stable closed subscheme of $X$, then it is clear that if $K$ acts freely on $X$ then it also acts freely on $Y$.
	Further, if $K$ acts freely on $X$, then so does any subgroup of $K$.

	One may verify the following: for $K=\alpha_p$ or $C_p$, a $K$-action is free if and only if for all $x\in X(\Bbbk)$, the maximal ideal sheaf $\mf_x$ of $x$ is not stable under $K$.

	In the following, denote by $\varphi_H:\Rep(H)\to\Ver_p$ the semisimplification functor of $\Rep(H)$.  
	It is explicitly given by the same formulas as $\Phi_H$ in (\ref{eqn oti explicit formula}), and in particular this allows us to naturally extend $\varphi_H$ to a functor on infinite-dimensional modules. 
	It is easy to see that $\varphi_H$ is still symmetric monoidal on the category of infinite-dimensional representations, and that Lemma \ref{OTI 0 iff proj and SES} also holds for $\varphi_H$.

	\begin{lemma}
	\label{lemma local global}
    	Let $H$ be either $\alpha_p$ or $C_p$, and let $X$ be an $H$-scheme with $\Fc$ an $H$-equivariant quasicoherent sheaf on $X$.  
    	Suppose that:
    	\begin{enumerate}
        \item $\varphi_H(H^i(X,\Fc))=0$ for $i>0$, and
        \item there exists a finite, $H$-stable affine covering $\{U_i\}$ of $X$ such that \linebreak $\varphi_H(\Gamma(U_I,\Fc))=0$ for all $U_I = \bigcap\limits_{i\in I}U_i$ with $|I|>0$.
    	\end{enumerate}
    	Then $\varphi_H(\Gamma(X,\Fc))=0$.
	\end{lemma}
	\begin{proof}
    We may consider the Cech complex $C^\bullet$ for $\Fc$ with respect to the affine covering given in (2), so that $H^i(C^\bullet)=H^i(X,\Fc)$.  
    By (2), $\varphi_H(C^i)=0$ for $i\geq 0$, and by (1), $\varphi_H(H^i(C^\bullet))=0$ for $i>0$.   
    By repeatedly applying Lemma \ref{OTI 0 iff proj and SES}, we may deduce that $\varphi_H(H^0(C^\bullet))=0$, which gives our result.
	\end{proof}

	\begin{lemma}
	\label{lemma freeness oti 0}
		Let $H=\alpha_p$ or $C_p$.  
		Suppose that $X=\Spec A$ is an affine $H$-scheme.  
		Then the following are equivalent:
    	\begin{enumerate}
        \item $H$ acts freely on $X$,
        \item there exists $f\in A$ such that $E^{p-1}(f)=1$ (resp.~$(u^*)^{p-1}(f)=1$), and
        \item $\varphi_H(\Gamma(X,\Fc))=0$ for every $H$-equivariant quasicoherent sheaf $\Fc$.
    	\end{enumerate}
	\end{lemma}

	We first state an easy lemma whose proof we leave as an exercise.
	\begin{lemma}
	\label{lemma proj crit kx/xp}
    	A module $M$ (of any dimension) over $\Bbbk[x]/x^p$ is projective (equivalently free) if and only if for all $v\in M$ with $xv=0$, there exists $w\in M$ such that $x^{p-1}w=v$.
	\end{lemma}

	\begin{proof}[Proof of Lemma \ref{lemma freeness oti 0}]
        For the implication $(2)\Rightarrow(1)$, if $H$ does not act freely on $X$ then there exists $x\in X(\Bbbk)$ such that $H$ stabilizes $\mf_x$.
        However, $f-f(x)\in\mf_x$, while, by assumption, the $H$-module generated by $f-f(x)$ contains $1$, a contradiction. 
        
        For $(3)\Rightarrow(2)$, by assumption, we have that $\Gamma(X,\Oc)$ is projective over $H$.  
        Thus by Lemma \ref{lemma proj crit kx/xp}, such an $f$ must exist.  
        For $(2)\Rightarrow(3)$, we apply Lemma \ref{lemma proj crit kx/xp}: let $s\in M=\Gamma(\Fc)$ be such that $E(s)=0$ (resp.~$u^*(s)=0$).  
        Then $E^{p-1}(fs)=s$ (resp.~$(u^*)^{p-1}(fs)=s$), implying $(3)$.
    
    	$(1)\Rightarrow(2)$: For the case of $H=\alpha_p$, let $x\in X(\Bbbk)$, and choose $f\in A$ such that $f(x)=0$ and $E(f)(x)\neq0$.  
    	Then we claim that $E^{p-1}(f^{p-1})(x)\neq0$.  Indeed, 
    	\[
    		E^{p-1}(f^{p-1})=(p-1)!E(f)^{p-1}+fg
    	\]
    	for some $g\in A$. 
    	Write $h:=E^{p-1}(f^{p-1})$.  
    	Then $E(h)=0$, so that $E^{p-1}(f^{p-1}/h)=h/h=1$ in $A_{h}$. 
    
 		For the case $H:=C_p$, we have that $C_p$ acts freely on $X(\Bbbk)$.  
 		Let $x\in X(\Bbbk)$ and choose $f\in A$ such that 
 		\[
 		f(x)
 		=
 		f(\sigma(x))
 		=
 		\cdots
 		=
 		f(\sigma^{p-2}(x))
 		=
 		0, 
 		\text{ and } 
 		f(\sigma^{p-1}(x))
 		\neq
 		0.
 		\]  
 		Then we see that $(u^*)^{p-1}(f)(x)\neq0$, so if we set $h:=(u^*)^{p-1}(f)$ we have $u^*(f/h)=1$.
 
 	It follows that in both cases, $\varphi_H(A_h)=0$ by $(2)\Rightarrow(3)$, where $h(x)\neq0$ for our chosen $x$.  
 	Since $x$ was arbitrary, we may do this on a finite $H$-stable open cover of $X$.  Applying Lemma \ref{lemma local global} completes the proof.
	\end{proof}

\subsection{Consequences for costandard modules} 
\label{sec: Consequences for costandard}

	Consider the flag variety $G/B$, and the action of $U_K\sub G$ on the left.
	\begin{lemma}
	\label{lemma free action}
    	The subgroup $U_K$ acts freely on the complement of the base point $eB$ inside of $G/B$.  
    	Consequently, so do the natural subgroups $\alpha_p,C_p\sub U_K$.
	\end{lemma}

	\begin{proof}
    	Every regular unipotent (resp.~nilpotent) element lies in a unique Borel subgroup (resp.~subalgebra) (see \cite[\S 3.7]{SteinbergConjClasses}).
    	Every nonidentity element of $U_K(\Bbbk)$ is regular, and any nonzero element of $\operatorname{Lie}(U_K)$ is also regular.  
    	From this we obtain freeness on the complement of $eB$.
	\end{proof}

	In the following, for $\lambda\in\mathfrak{X}$, write $\Oc(\lambda)$ for the sheaf of sections of the line bundle $G\times_{B}\Bbbk_{\lambda}$, and write $H^i(\lambda):=H^i(G/B,\Oc(\lambda))$.  
	In particular, for $\lambda\in\mathfrak{X}_+$ we have $H^0(\lambda)=\nabla_{\lambda}$ and $H^i(\lambda)=0$ for $i>0$ by Kempf vanishing (\cite[I, Chapter 4]{JantzenBook}).

	\begin{lemma}
	\label{lemma translation invariance}
    	Let $\lambda,\mu\in\mathfrak{X}$ be such that {$\Phi_H(H^i(\lambda))=0$} for $i>0$ and both $\lambda+p\mu,\mu$ lie in $\mathfrak{X}_+$.  
    	Then we have an isomorphism in $\Ver_p$
    	\[
			\Phi_H(H^0(\lambda))
			\xrightarrow{\sim}
			\Phi_H(\nabla_{\lambda+p\mu}).
    	\]
    	If $H\cong\alpha_p$, then the morphism is $\Phi_H(\gf)$-equivariant and of weight $2p\mu(\rho^\vee)$ with respect to $T_{K}$.  
	\end{lemma}

	\begin{proof}
		Since the statement is clear for $\mu=0$, we will assume $\mu\neq0$, and let $s\in\Gamma(\Oc(\mu))=\nabla_{\mu}$ be nonzero of weight $\mu$, so, in particular, $s(eB)\neq0$.

    	For the case $H=\alpha_p$, we have a short exact sequence
    	\begin{equation}
    	\label{eqn SES alphap}
    		0
    		\to 
    		\Oc(-p\mu)
    		\xrightarrow{s^p}
    		\Oc
    		\to 
    		\Oc_{Z}
    		\to
    		0
    	\end{equation}
    	where $Z$ is the vanishing subscheme of $s^p$.  
    	The morphism $s^p$ is $\alpha_p$-equivariant, meaning that $Z$ is $\alpha_p$-stable, and thus $\alpha_p$ acts freely on $Z$ by Lemma \ref{lemma free action}.  

    	For the case $H=C_p$, we instead consider the morphism
    	\begin{equation}
    	\label{eqn SES Cp}
    		0
    		\to 
    		\Oc(-p\mu)
    		\xrightarrow{s\sigma(s)\cdots\sigma^{p-1}(s)}
    		\Oc
    		\to
    		\Oc_{Z}
    		\to
    		0.
    	\end{equation}
    	In this case $Z$ denotes the vanishing subscheme of $s\sigma(s)\cdots\sigma^{p-1}(s)$, and we see it is $C_p$-stable. 
    	Since $\sigma^i(s)$ is non-vanishing at $eB$ for all $i$, $eB\notin Z$, and thus $C_p$ will act freely on $Z$ by Lemma \ref{lemma free action}. 

    	In either case, we may tensor our short exact sequence with the $G$-equivariant line bundle $\Oc(\lambda+p\mu)$ and obtain:
    	\begin{equation}
    	\label{eqn ses line bundles}
        	0
        	\to 
        	\Oc(\lambda)
        	\to 
        	\Oc(\lambda+p\mu)
        	\to 
        	\Oc_Z(\lambda+p\mu)
        	\to 
        	0.
    	\end{equation}
    	Since $\Oc(\lambda+p\mu)$ has vanishing higher cohomology, we have $H$-equivariant isomorphisms $H^{i+1}(\lambda)\cong H^{i}(G/B,\Oc_{Z}(\lambda+p\mu))$ for $i>0$, implying by assumption that $\varphi_H(H^{i}(G/B,\Oc_{Z}(\lambda+p\mu)))=0$ for $i>0$. 
    	Because $Z$ is quasiprojective, it admits an $H$-stable affine covering.
    	Indeed, for any $x\in Z(\Bbbk)$, $H\cdot x$ is a finite subscheme of $Z$ and is thus contained in an affine open subvariety $U\sub Z$.  
    	Then $\bigcap\limits_{h\in H(\Bbbk)}h\cdot U$ will be an $H$-stable affine open containing $x$.  
    
    	We may now apply Lemmas~\ref{lemma local global} and \ref{lemma freeness oti 0} to deduce that $\varphi_H(H^0(Z,\Oc_Z(\lambda+p\mu)))=0$. 
    	Our sequence (\ref{eqn ses line bundles}) gives the exact sequence
    	\[
    		0
    		\to 
    		H^0(\lambda)
    		\to 
    		H^0(\lambda+p\mu)
    		\to 
    		H^0(G/B,\Oc_{Z}(\lambda+p\mu))
    		\to 
    		H^1(\lambda)\to0,
    	\]
    	and the last two terms are projective over $H$.  
    	By passing to exact triangles, we learn that the map $H^0(\lambda)\to H^0(\lambda+p\mu)$ gives an isomorphism $H^0(\lambda)\cong H^0(\lambda+p\mu)$ in $\Rep(H)^{st}$, and thus an isomorphism:
    	\[
    		\Phi_H (H^0(\lambda))\cong\Phi_H (H^0(\lambda+p\mu)).
    	\]
    	For the final statement in the case $H\cong\alpha_p$, $s^p$ is a morphism of weight $2p\mu(\rho^\vee)$, and is $\Phi_H(\gf)$-equivariant because derivations annihilate any $p$th power. 
	\end{proof}

	\begin{cor}
	\label{cor vanishing geo}
    	If $H^i(\lambda)=0$ for all $i$, and both $\lambda+p\mu,\mu$ lie in $\mathfrak{X}_+$, then $\Phi_H(\nabla_{\lambda+p\mu})=0$.  
	\end{cor}

\subsection{Vanishing along walls via geometry}
\label{sec: vanishing p-sing geom}

	\begin{definition}
    	Define the fundamental polyhedron of $G$ to be 
    	\[
    		\Pc_0
    		:=
    		\ol{W_{f}\cdot A_0}
    		=
    		W_{f}\cdot\ol{A_0}.
    	\]
	\end{definition}

	Recall from \cite[II, Corollary 5.5]{JantzenBook} that for $\lambda\in \ol{A_0}\cap\mathfrak{X}_+$ we have, for $w\in W_{f}$,
	\[
		H^i(w\cdot \lambda)
		=
		\begin{cases}
    		L_\lambda & \text{ if }i=\ell(w)
    		\\ 
    		0 & \text{ else,}
		\end{cases}
	\]
	and for all other $\lambda\in\ol{A_0}$ we have $H^\bullet(w\cdot\lambda)=0$. 
	We denote by $\partial A_0$ the union of boundary walls of the fundamental alcove.  

	\begin{lemma}
	\label{lemma numerical vanishing}
    	Suppose that $\beta^\vee\in R^\vee$ is such that $\langle\varpi,\beta^\vee\rangle=1$ for some fundamental dominant weight $\varpi$.  
    	Then for $\lambda\in\mathfrak{X}_+ \cap \ol{A_0}$ with $\langle\lambda+\rho,\beta^\vee\rangle=p$, we have $\Phi_H(\nabla_\lambda)=0$.
	\end{lemma}
	\begin{proof}
    	In this case, $\lambda-p\varpi$ lies in the intersection of $W_{f}\cdot(\partial A_0)$  with the interior of $\Pc_0$, so we may conclude by Corollary \ref{cor vanishing geo}.
	\end{proof}

	\begin{thm}
	\label{thm vanishing along walls fund box}
    	If $G$ is not of type $\textbf{G}_2,$ $\textbf{F}_4$, or $\textbf{E}_8$, and  $\lambda\in\partial A_0$ and $w\in W_{f}$, then \linebreak $\Phi_H (H^{\ell(w)}(w\cdot \lambda))=0$.
	\end{thm}

	\begin{proof}
    	It suffices to show this for $\lambda\in \partial A_0\cap\mathfrak{X}_+$, so that $\langle\lambda+\rho,\alpha_0^\vee\rangle=p$.  
    	Because $G$ is not of type $\textbf{G}_2$, $\textbf{F}_4$, or $\textbf{E}_8$, we have that $\alpha_0^\vee$ satisfies the hypothesis of Lemma \ref{lemma numerical vanishing} (see Plates I-IX after \cite[Chapter VI]{Bourbaki}), so we are done. 
	\end{proof}

	\begin{cor}
    	If $G$ is not of type $\textbf{G}_2,\textbf{F}_4$, or $\textbf{E}_8$, then $\Phi_H(\nabla_\lambda)=0$ for all $p$-singular $\lambda\in\mathfrak{X}_+$.
	\end{cor}
	\begin{proof}
    Let $\lambda\in\mathfrak{X}_+$ be $p$-singular.  
    Since $\overline{A_0}$ is a fundamental domain for $W$, we have
\[
    W_{f}\cdot (\overline{A_0} \cap \mathfrak{X})+p\mathfrak{X}=\mathfrak{X}.\]  
    Hence there exists $\mu\in\mathfrak{X}$ such that $\lambda+p\mu\in(W_{f}\cdot\partial A_0)$.  
    Let $\mu'\in\mathfrak{X}_+$ such that $\mu+\mu'\in\mathfrak{X}_+$.  
    By Theorem \ref{thm vanishing along walls fund box}, and Lemma \ref{lemma translation invariance},  we have $\Phi_H(\nabla_{(\lambda+p\mu)+p\mu'})=0$. 
    On the other hand, again by Lemma \ref{lemma translation invariance}, $\Phi_H(\nabla_{\lambda})\cong\Phi_H(\nabla_{\lambda+p(\mu+\mu')})$ so we obtain the result.    
	\end{proof}

\section{Proof of Theorem B}
\label{Sec: Proof Thm B}

	We continue with our setup from \S\ref{sec: Alg grps}, \ref{sec prin SL2}.  
	We take $H=(U_K)_{1}\cong\alpha_p$, so that it is a regular subgroup of $G$.
	Our aim is to show that on the extended principal block, our functor takes values in super vector spaces within the larger category $\Ver_p$:
	\begin{equation}
	\label{eqn lands in svec}
    	\Phi_H(\Rep ^{\ext}_0(G))\sub \operatorname{sVec}.
	\end{equation}
	We will prove this result by first finding the image of $\Rep_0^{\ext}(G)$ in the stable category, since the OTI functor factors through the stable category.  
	In fact, we will show that our functor $\Phi_H$ factors through a semisimple subcategory of the stable category in this case, and we will use this to prove certain exactness properties of $\Phi_H$.
	
	\begin{remark}
		Recall from \S\ref{section oti blocks} that equation (\ref{eqn equiv rep cats}) implies that we may view $\Phi_H$ as a functor from $\Rep(G)$ to $\Ver_p\boxtimes\Ver_p(G)$.  
		Using Lemma \ref{lemma trans OTI}, the results of this section in fact show that:
		\begin{equation}
		\label{eqn lands in svec tensor verpg}
			\Phi_H(\Rep(G))\sub\operatorname{sVec}\boxtimes\Ver_p(G).    
		\end{equation}
		More explicitly, if $\lambda_0\in A_0\cap\mathfrak{X}$ and $M\in\Rep_{\lambda_0}(G)$, then $\Phi_H(M)$ has composition factors $\Phi_H(L_{\lambda_0})$ and $\Pi\Phi_H(L_{\lambda_0})$.   
	\end{remark}

\subsection{The Capricorn group}
\label{sec: Capricorn}

	\begin{definition}
  		Define the Capricorn group $C$ to be the algebraic group $\Gbb_m\ltimes\alpha_p$, where $t\cdot E= t^{-2}E$ for $t\in\Gbb_m(\Bbbk)$, and $E$ is as in \S\ref{sec: Reps Cp alphap}.
	\end{definition}
	The following is obvious. 
	\begin{lemma}
   		The data of a $C$-module $V$ is the same thing as a graded vector space $V=\bigoplus_{i \in \mathbb{Z}} V_i$ such that
   		\begin{enumerate}
       		\item $\Gbb_m$ acts with weight $i$ on $V_i$;
       		\item $E$ is an endomorphism of $V$ with $E^p=0$;
       		\item $E(V_i)\sub  V_{i-2}$.
   		\end{enumerate}
	\end{lemma}

	Observe that given a choice of torus and Borel subgroup $T\sub B\sub SL_2$, we have that $C\cong B_1T$. 
	In particular, if we write $C_K:=(U_K)_1T_K$, then we have a morphism $C\to C_K$, which is either an isomorphism or a two-fold cover.  
	We call $C_K$ a principal Capricorn subgroup of $G$.

\subsubsection{Stable module category of $C$}
\label{Sec: Stable module cat}

	We define $M(a, d)$ to be the indecomposable $C$-module of dimension $d+1$, with lowest weight $a\in\Zbb$, which we depict below:

	\begin{center}
    \begin{tikzcd}
  		|[label=\textcolor{blue}{a}]|{\Bbbk} 
  		& |[label = \textcolor{blue}{a + 2}]|{\Bbbk} \arrow[l] 
  		& \cdots \arrow[l]
  		& |[label=\textcolor{blue}{a+2d}]|{\Bbbk} \arrow[l] 
		\\
	\end{tikzcd}
	\end{center}
	The $E$-action is given by the arrows, while the weights are given above in blue.
	We note that every indecomposable $C$-module is isomorphic to $M(a,d)$ for some $a\in\Zbb$, $0\leq d\leq p-1$.

	Consider the natural functor $(-)^{st}:\Rep (C)\to\Rep (C)^{st}$ where $\Rep (C)^{st}$ is the stable module category of $\Rep (C)$, see  \S\ref{sec: stable module category and semisimplification}. 
 	We recall that $\Rep (C)^{st}$ is tensor triangulated, with a symmetric monoidal structure inherited from $\Rep(C)$.  
    Further, $M^{st}=0$ iff $M$ is projective, and $(-)^{st}$ takes short exact sequences to exact triangles.
 	We will suppress the superscript when it is clear from context, writing $M$ in place of $M^{st}$.

	Note that a module is projective over $C$ if and only if it is projective over $\alpha_p$. 
	The indecomposable module $M(a,d)$ is projective over $\alpha_p$ if and only if $d=p-1$.  
	Now, consider the short exact sequence
	$$ 
		M(a,d) \hookrightarrow M(a, p-1) \twoheadrightarrow M(a + 2(d+1), p - d-2).
	$$

	Passing to $\Rep (C)^{st}$ we have $M(a, p - 1) \cong 0$, so we obtain the distinguished triangle
	$$
		M(a,d) \rightarrow 0 \rightarrow M(a + 2(d+1), p - d-2)\rightarrow M(a,d)[1]
	$$
	where $[1]$ denotes the shift in $\Rep (C)^{st}$. 
	Therefore we have 
	\begin{equation}  
	\label{shift}   
		M(a,d)[1] \cong M(a + 2(d+1), p -d-2).
	\end{equation}
	This tells us how to calculate shifts in the stable category.
	\begin{lemma}
	\label{lemma shifts of O}
    	For $a, m\in\Zbb$, we have the following isomorphisms in $\Rep (C)^{st}$:
    	\[
    		M(a,0)[m]\cong \begin{cases}
        	M(a + mp - p+2, p - 2) & \text{ if } m \text{ is odd,}\\
        	M(a + mp, 0) & \text{ if $m$ is even}.
    	\end{cases}
    	\] 
	\end{lemma}
	\begin{proof}
    	This follows from a straightforward induction argument using equation (\ref{shift}).
	\end{proof}

	The following computation, while simple, is critical to the main argument of this section.
	\begin{lemma}
	\label{lemma homs st cat}  
		Let $a, m,m'\in\Zbb$.  
		We have
    	$$ 
    	\operatorname{Hom}_{\Rep (C)^{st}}(M(a, 0)[m], \; M(a, 0)[m']) 
    	= 
    	\begin{cases}
        	\Bbbk & \text{if } m = m';
        	\\
        	0 & \text{if } m \neq m'.
    	\end{cases}
    	$$
	\end{lemma}
	\begin{proof} 
        By Lemma \ref{lemma shifts of O}, for any integer $m$, we
  		have isomorphisms
  		\begin{gather}
			M(a,0)[2m] \cong M(a + 2mp,0) \label{eq:sh1}\\
			M(a,0)[2m+1]  \cong M(a + 2mp + 2, p-2). \label{eq:sh2}
		\end{gather}
		For a $C$-module $V$, denote by $\operatorname{supp}(V)\sub\Zbb$ the set $\{n\in\Zbb:V_n\neq 0\}$. 
		By Lemma \ref{lemma shifts of O}, we compute that
    	\[   
			\operatorname{supp}M(a + 2mp,0)=\{a + 2mp\},
    	\]
  		and
    	\[    
			\operatorname{supp} M(a + 2mp + 2, p-2)
			=
			\{a +2mp+2,a +2mp+4,\dots,a+2(m+1)p-2\}.
    	\]
   		From this we conclude that the supports of any distinct modules of the form $M(a+2mp,0)$ or $M(a+2mp + 2, p-2)$ are disjoint. 
    	Combining this with \eqref{eq:sh1} and \eqref{eq:sh2} proves that
		\[
     		\Hom_{\Rep(C)^{st}}(M(a, 0)[m], \; M(a,0)[m'])
     		=
     		0 
     		\quad 
     		\text{if $m\neq m'$.}
		\]
     	It remains to observe that $\End_{\Rep (C)}(M(a,d))\cong\Bbbk$
        if $0 \le d < p-1$, and we are done.
	\end{proof}

	\begin{example}
		If we consider the costandard modules $\nabla_n$ in the extended principal block of $SL_2$, restrict them to the principal Capricorn subgroup $C \subset B$, and consider  their images in $\Rep (C)^{st}$ we have isomorphisms:
  		\begin{gather*}
    		\nabla_{mp} 
    		\cong 
    		M(mp,0), 
    		\\
    		\nabla_{mp-2} 
    		\cong 
    		M((m-2)p+2,p-2).
  		\end{gather*}
  		(This can be checked by hand, or the reader can wait for Lemma \ref{lemma prep}.) 
  		One observes that every weight space is present in at most one of these modules. 
  		See Figure \ref{fig:supports} for the case $p=3$.
	\end{example}

	\begin{figure}[htbp]
  	\centering
  	\begin{tikzpicture}
	\matrix (m) [matrix of nodes,
  		nodes in empty cells,
  		column 1/.style={nodes={text width=6mm, align=center}},
  		nodes={draw, minimum width=6mm, minimum height=6mm, anchor=center},
  		row sep=-\pgflinewidth,
  		column sep=-\pgflinewidth,
		]{
 		& -1 & 0 & 1 & 2 & 3 & 4 & 5 & 6 & 7 & 8 & 9 & 10 & 11 & 12 & 13 & 14 & 15 & 16 \\
		$\nabla_0$    &  & $\bullet$ &  &  &  &  &  &  &  &  &  &  &  &  &  &  &  &    \\
		$\nabla_1$    & $\bullet$ &  & $\bullet$ &  &  &  &  &  &  &  &  &  &  &  &  &  &  &    \\
		$\nabla_3$    &  &  &  &  & $\bullet$ &  &  &  &  &  &  &  &  &  &  &  &  &  \\
		$\nabla_4$    &  &  &  & $\bullet$ &  & $\bullet$ &  &  &  &  &  &  &  &  &  &  &  &   \\
		$\nabla_6$    &  &  &  &  &  &  &  & $\bullet$ &  &  &  &  &  &  &  &  &  &  \\
		$\nabla_7$    &  &  &  &  &  &  & $\bullet$ &  & $\bullet$ &  &  &  &  &  &  &  &  &  \\
		$\nabla_9$    &  &  &  &  &  &  &  &  &  &  & $\bullet$ &  &  &  &  &  &  &  \\
		$\nabla_{10}$ &  &  &  &  &  &  &  &  &  & $\bullet$ &  & $\bullet$ &  &  &  &  &  &  \\
		$\nabla_{12}$ &  &  &  &  &  &  &  &  &  &  &  &  &  & $\bullet$ &  &  &  &  \\
		$\nabla_{13}$ &  &  &  &  &  &  &  &  &  &  &  &  &   $\bullet$ &  &  $\bullet$ &  &  & \\
		$\nabla_{15}$ &  &  &  &  &  &  &  &  &  &  &  &  &  &  &  &  & $\bullet$ &   \\
		$\nabla_{16}$ &  &  &  &  &  &  &  &  &  &  &  &  &  &  &  & $\bullet$ &  & $\bullet$  \\
		};
	\end{tikzpicture}
    \caption{Supports of the OTI functor applied to costandard modules
      in the extended principal block for $SL_2$ and $p=3$.}
     \label{fig:supports}
   \end{figure}

	\subsubsection{A semisimple subcategory of $\Rep(C)^{st}$} In the following, by a triangulated full subcategory $\mathcal{T}$ of $\Rep (C)^{st}$, we mean a full additive subcategory that is closed under isomorphisms and shifts, and has the property that if
	\[
		X\to Y\to Z\to X[1]
	\]
	is a distinguished triangle in $\Rep (C)^{st}$ and $X,Y$ lie in $\mathcal{T}$, then so does $Z$ (see \cite[\S1.5]{Neeman}).
    
 	Inside $\Rep (C)^{st}$ we consider the triangulated full category generated by $M(0,0)$ and $M(2-p,p-2)$:
 	\[
		\Rep (C)^{st}_p 
		:= 
		\langle M(0,0), M(2-p,p-2) \rangle 
		= 
		\langle M(0,0), M(-p,0) \rangle.
	\]
	(The second equality follows because $M(-p,0)[1] = M(2-p,p-2)$
        by \eqref{shift}.) Note that $\Rep(C)^{st}_p$ is a symmetric
        monoidal subcategory of $\Rep(C)^{st}$ because
        \[
          M(0,0)\otimes M(-p,0) \cong M(-p,0), \quad
          M(-p,0) \otimes M(-p,0) \cong  M(-2p,0) \cong  M(0,0)[-2],
        \]
        where we have used Lemma \ref{lemma shifts of O} and that the shift functor respects $\otimes$.

	\begin{prop} 
	\label{prop:equiv to svec}
  		The category $\Rep (C)^{st}_p$ is equivalent (as an additive category) to graded super vector spaces. More precisely, there exists a unique equivalence
  		\[
			\mathcal{E} : \Rep (C)^{st}_p  \to \operatorname{gr}_{\Zbb}\sVec
		\]
		satisfying the following two conditions:
		\begin{enumerate}
			\item We have  $M(0,0) \mapsto \Bbbk^{1|0}$ and $M(2-p,p-2) \mapsto \Pi \Bbbk^{1|0} = \Bbbk^{0|1}$, where $\Bbbk^{1|0}$ (resp. $\Bbbk^{0|1}$) denotes the purely even (resp.~odd) super vector space of dimension 1 in degree zero.
			\item For all $M \in \Rep (C)^{st}_p$ we have a natural isomorphism
  				\[
					\mathcal{E}(M[m]) 
					\cong 
					\Pi^m \mathcal{E}(M)[m]
    			\]
				where $[m]$ denotes the standard grading shift functor on graded modules (i.e. $M[m]_{j} = M_{m+j}$). 
  		\end{enumerate}
	\end{prop} 

	\begin{proof} 
		In this proof we will find it easier to work with the description
		\[
			\Rep (C)^{st}_p = \langle M(0,0), M(-p,0) \rangle.
		\]
		Write $\mathcal{C}$ for the full additive subcategory of $\Rep (C)^{st}$ generated by the objects $M(0,0)[m], M(-p,0)[n]$ for $m,n\in\Zbb$, and which is closed under isomorphisms.  
		We would like to show that $\Cc$ is a triangulated full subcategory of $\Rep (C)^{st}$, which would imply that $\mathcal{C}=\Rep(C)^{st}_p$.

		By construction, $\mathcal{C}$ is additive, closed under isomorphisms and closed under shifts.  
		Thus it suffices to show if $X,Y$ lie in $\Cc$, and we have a triangle
		\[
			X\to Y\to Z\to X[1],
		\]
		then $Z$ also lies in $\Cc$. 
		However, we know by Lemma \ref{lemma homs st cat} that
                maps between any two shifts of
                $M(0,0)$ (resp. $M(-p,0)$) are either
                isomorphisms or 0. On the other hand, there are no
                non-zero maps in either direction between any shift of
                $M(0,0)$ and any shift of $M(-p,0)$, because the
                weights are of different parity (see Lemma \ref{lemma shifts of O}).
		Thus it is easy to check that $Z$ also lies in $\Cc$.
  
		It follows that $\Rep (C)^{st}_p$ is semisimple and
                we can construct a functor out of it by specifying the
                images of simple objects, which are uniquely
                determined by conditions (1) and (2).
	\end{proof}

	\begin{remark} 
	\label{rem:stable equivalence}
  		Proposition \ref{prop:equiv to svec} implies that $\Rep (C)^{st}_p$ is a semisimple, symmetric tensor category.  
    \end{remark}
    
    \begin{remark}
     In fact, $\mathcal{E}$ is an equivalence of symmetric tensor categories, however we will not need this
                fact. Because $\mathcal{E}$ is an equivalence, we only
                need to check it is a symmetric monoidal functor.  It is straightforward to give a monoidal structure to $\mathcal{E}$ (see \cite[Definition 2.4.1]{EGNO}), and the fact that $\mathcal{E}$ is symmetric is left as an exercise for the
                reader.  Alternatively, one can show that $\Rep(C)^{st}_p$ is equivalent to $\operatorname{gr}_{\Zbb}\sVec$ as a symmetric tensor category via an argument using super Tannakian formalism.
	\end{remark}

	\begin{remark} 
		Below it will be important to consider the functor
  		\[
			\Rep (C)^{st}_p \to \operatorname{gr}_{\Zbb}\Vec
		\]
		obtained by composing $\mathcal{E}$ above with the functor that forgets the $\mathbb{Z}/2\mathbb{Z}$ grading on the super vector space. 
		One calculates easily that this functor maps
		\[
			M(mp,0) 
			\mapsto 
			\Bbbk[m] 
			\quad 
			\text{and} 
			\quad 
			M(mp+2,p-2) 
			\mapsto 
			\Bbbk[m+1].
  		\]
		By drawing pictures similar to those in Figure 1, one sees that this functor ``divides the centre of mass by $p$, and negates the result''.
	\end{remark}

\subsubsection{Restriction to the principal Capricorn subgroup}
\label{sec: Principal capricorn}

	Recall that we have the quotient morphism $C\to C_K$, where $C_K$ is a principal Capricorn subgroup of $G$ (see \S\ref{sec: Capricorn}), meaning we have a full tensor subcategory $\Rep(C_K)\sub\Rep(C)$. 
	The kernel of $C\to C_K$ is either trivial or $\mu_2$ (the subgroup scheme of $\Gbb_m$ of elements of order 2), which both have semisimple representation theory.  
	Thus a $C_K$-module is projective if and only if it is projective as a $C$-module.  
	In particular, we have a natural embedding of tensor triangulated categories $\Rep(C_K)^{st}\sub\Rep(C)^{st}$. 

	In what follows, we will write 
	\begin{equation}
	\label{eqn Phi st}
		\Phi_H^{st}:\Rep (G)\to\Rep(C_K)^{st}\sub \Rep (C)^{st}
	\end{equation}
	for the functor given by $\Phi_H^{st}(M)=(\operatorname{Res}_{C_K}^{G}M)^{st}$.  
	We claim that we have the following commutative diagram:
	\[
	\xymatrix{
		\Rep (G) \ar@{^{(}->}[r] \ar@/^2pc/[rrr]^{\Phi_H^{st}}\ar[d]_{\Phi_H} & D^b(\Rep(G)) \ar[r]& D^b(\Rep(C)) \ar[r] & \Rep (C)^{st} \ar[d]^{ss} \,
		\\ 
		\operatorname{gr}_{\Zbb}\Rep _{\operatorname{Ver}_p}(C_G(H))\ar[rrr] & && \operatorname{gr}_{\Zbb}\Ver_p.
	}
	\]
	The functor $D^b(\Rep(C))\to\Rep(C)^{st}$ is the quotient functor coming from the description of the stable category as a quotient of the derived category by perfect complexes \cite[Theorem 2.1]{RickardStable}.\footnote{Rickard proves this result for module categories over self-injective algebras, but it is easy to see that his proofs carry over to $\Rep (C)$.} 
	The right vertical arrow is semisimplification (because $\operatorname{gr}_{\Zbb}(\Ver_p)$ is the semisimplification of $\Rep(C)$), and the bottom horizontal arrow is forgetting the $C_G(H)$-action.

	\begin{prop}
	\label{prop stable cat im}
    	For a block $\Rep_{\lambda}(G)$ of $\Rep^{\ext}_0(G)$, there exists $a\in\{ -p, 0 \}$ such that 
    	\[
    	\Phi_{H}^{st}(\Rep_\lambda(G))\sub\langle M(a,0)\rangle.
  		\]
		In particular, the image of $\Rep ^{\ext}_0(G)$ under $\Phi_H^{st}$ in $\Rep (C)^{st}$ lies in $\Rep (C)^{st}_p$.
	\end{prop}

	We start with a lemma. 

	\begin{lemma}
	\label{lemma prep}
    	\begin{enumerate}
    		\item For $\lambda_0\in (W^{\ext}\cdot 0)\cap A_0$ and $x \in {}^f W$ we have
      			\[
        			\Phi_H^{st}(\nabla_{x \cdot \lambda_0}) 
        			\cong 
        			\Phi_H^{st}(\nabla_{\lambda_0})[\ell(x)].
        		\]
        	\item For $\lambda,\mu\in\mathfrak{X}_+$, we have 
        		\[
        			\Phi_H^{st}(\nabla_{\lambda+p\mu})
        			\cong
        			\Phi_H^{st}(\nabla_{\lambda})\otimes M(2p\mu(\rho^\vee),0).
        		\]
      		\item For $\lambda_0= y \cdot 0 \in (W^{\ext}\cdot 0)\cap A_0$, we have 
        		\[ 
					\Phi_H^{st}(\nabla_{\lambda_0})\cong 
					\begin{cases} 
						M(0,0) & \text{if $\epsilon(y) = 1$}, 
						\\
      					M(2-p,p-2) & \text{if $\epsilon(y) = -1$,} 
      				\end{cases} 
				\]
			(recall the sign character $\epsilon$ on $W^{\ext}$ was defined in \S\ref{sec:Weyl groups and alcoves}).
    	\end{enumerate}  
	\end{lemma}
	\begin{proof}
		For statement (1), it is enough to prove that if $xs > x$ for $s\in S$, then we have
    	\[
			\Phi_H^{st}(\nabla_{xs \cdot \lambda_0}) 
			\cong 
			\Phi_H^{st}(\nabla_{x \cdot \lambda_0})[1].
		\]
		This isomorphism was essentially already stated in the proof of Corollary \ref{cor:shifting}, and follows by considering the exact triangle obtained from the wall-crossing exact sequence (\ref{eqn wall crossing SES}).
		The claim in (2) is the stable category version of Lemma \ref{lemma translation invariance}, where the isomorphism is induced by the map on global sections of the short exact sequence (\ref{eqn ses line bundles}).

		For the last statement, we can find $\mu$ dominant such that $y = t_{-\mu} x$ and $x\in W$.
		Hence
		\begin{equation} 
		\label{eq:lambda0id}
  			\lambda_0 + p\mu 
  			= 
  			x \cdot 0 
  			\quad 
  			\text{with $x \in W$},
		\end{equation}
		and, as explained in \S\ref{sec:Weyl groups and alcoves},
		\begin{equation}
		\label{eq:signs agree}
			\epsilon(y) 
			= 
			\epsilon(x). 
		\end{equation}
		Now, by part (2), we have
		\[
		\Phi_H^{st}(\nabla_{\lambda_0}) 
		\cong 
		\Phi_H^{st}(\nabla_{x \cdot 0}) \otimes M(-2p\mu(\rho^\vee),0).
		\]
		By a classic formula for the length function \cite[Proposition I.23]{IM}, we have
		\[
			\ell(x) 
			= 
			\ell(t_\mu) 
			= 
			2\mu(\rho^\vee).
		\]
		If $\ell(x)$ is even then, by Lemma \ref{lemma shifts of O},
		\[
			\Phi_H^{st}(\nabla_{\lambda_0}) 
			\cong 
			\Phi_H^{st}(\nabla_{x \cdot 0}) \otimes M(-2p\mu(\rho^\vee),0) 
			\cong 
			M(0,0)[\ell(x)] \otimes M(0,0)[-\ell(x)]
			\cong 
			M(0,0).
		\]
		If $\ell(x)$ is odd then (again using Lemma \ref{lemma shifts of O})
		\[
  			\Phi_H^{st}(\nabla_{\lambda_0}) 
  			\cong 
  			\Phi_H^{st}(\nabla_{x \cdot 0}) \otimes M(-2p\mu(\rho^\vee),0) 
  			\cong 
  			M(0,0)[1] \otimes M(-p,0) \cong M(2-p,p-2).
		\]
		Statement (3) now follows from \eqref{eq:signs agree}.
	\end{proof}

	\begin{proof}[Proof of Proposition \ref{prop stable cat im}] 
		The category $\Rep ^{\ext}_0(G)$ is a union of blocks $\Rep_{\lambda_0}(G)$, for \newline $\lambda_0\in (W^{\ext}\cdot 0)\cap A_0$, see \S\ref{sec:Weyl groups and alcoves}.  
		By (3) of Lemma \ref{lemma prep} we may write
		\[
			\Phi_H^{st}(\nabla_{\lambda_0})\cong M(a,0)[m],
		\]
		for some $a \in \{ -p, 0 \}$ and $m\in\Zbb$. 
		(Recall that $M(2-p,p-2) \cong M(-p,0)[1]$ by \eqref{shift}.)
		By (1) of Lemma \ref{lemma prep} it follows that for any dominant $\mu\in W\cdot \lambda_0$,
		\[
			\Phi_H^{st}(\nabla_{\mu})\in\langle M(a,0)\rangle. 
		\]
		The category $D^b(\Rep_{\lambda_0}(G))$ is generated by the costandard modules $\nabla_{\mu}$ for dominant $\mu\in W\cdot\lambda_0$.
  		Since the composition $D^b(\Rep_\lambda(G))\to D^b(\Rep(C))\to\Rep(C)^{st}$ is triangulated, we obtain that
		\[
  			\Phi_H^{st}(\Rep_{\lambda_0}(G))\sub\langle M(a,0)\rangle. \qedhere
  		\]
	\end{proof}

	This allows us to establish Theorem B. 
	\begin{thm}
	\label{thm land in sVec} 
		$\Phi_H(\Rep _{0}^{\ext}(G))\sub\operatorname{sVec}$.
	\end{thm}

	\begin{proof} 
		We can factor $\Phi_H$ as the composition
  		\[
			\Rep _{0}^{\ext}(G) 
			\stackrel{\Phi_H^{st}}{\longrightarrow} 
			\Rep (C)^{st}  
			\stackrel{\Res}{\longrightarrow} 
			\Rep(H)^{st} 
			\stackrel{ss}{\longrightarrow} 
			\Ver_p
		\]
		where the second arrow denotes the functor induced by restriction to $H = \alpha_p$ and the third arrow denotes semisimplification. 
		We know from Proposition \ref{prop stable cat im} that $\Phi_H^{st}$ lands in $\Rep (C)_p^{st}$, whose indecomposable objects are all of the form $M(a,d)$ where $d \in \{ 0, p-2 \}$. 
		These objects map to $L_0$ and $L_{p-2}$ in $\Ver_p$, which completes the proof.
	\end{proof}

    \begin{remark}
        Using Lemma \ref{lemma prep}, one may obtain formulae for $\Phi_H(\nabla_{x\cdot 0})$ for $x\in W^{\ext}$.  
        Moreover, using this along with Lemma \ref{lemma trans OTI}, one may obtain formulae for $\Phi_H(\nabla_{x\cdot\lambda_0})$ for any $\lambda_0\in A_0\cap\mathfrak{X}$ and $x\in W^{\ext}$.   
    \end{remark}

\section{Further results}
\label{Sec: Further}

	In this final section, we assume throughout that $H=\alpha_p$.

\subsection{OTI functor via complexes of tilting modules}
\label{sec: Complexes of tilting modules}

	In this section we reinterpret our functor via complexes of tilting modules. 
	This interpretation gives an alternative proof of most of Theorems A and B. 
	We will also use this in the next section to connect $\Phi_H$ with Gruber's theory of singular modules for $G$.

\subsubsection{Minimal complexes}
\label{Sec: minimal complexes}

	Recall that for any Krull-Schmidt additive category $\mathcal{A}$ one may speak of minimal complexes. 
	More precisely, any complex $M=M^{\bullet} \in K^b(\mathcal{A})$ admits a summand 
	\[	
		M^{\min} \subset M
	\] 
	which is isomorphic to $M$ in $K^b(\mathcal{A})$ and may be obtained from $M$ by repeatedly deleting contractible summands, until this is no longer possible. 
	In particular, if we decompose each term into indecomposable summands
	\[
		M^i = \bigoplus T_j^{\oplus m^i_j},
	\]
	then we may view the differential on our complex as a matrix of morphisms between indecomposable modules. 
	Our complex is minimal if and only if no entries of these matrices are isomorphisms. 
	(For more on minimal complexes, see e.g. \cite[\S 2.1]{GruberMinimal}, \cite[\S 6.1]{EW} and \cite{Krause}.)

\subsubsection{Negligible tilting modules}
	The category $\Tilt$ is a Karoubian, rigid, symmetric monoidal category, and thus we may speak of its ideal $\Tilt_{neg}$ of negligible morphisms.  
	As usual, we say that $T\in\Tilt$ is negligible if $\operatorname{id}_{T}$ lies in $\Tilt_{neg}$.  
	Recall that $T$ is negligible if and only if its indecomposable summands are negligible, and an indecomposable tilting module is negligible if and only if its dimension is divisible by $p$ {(see \S\ref{sec: stable module category and semisimplification})}.

	\begin{lemma}
	\label{lemma negl tilts}
		For $\lambda\in\mathfrak{X}_+$, the following are equivalent:
		\begin{enumerate}
    	\item $T_{\lambda}$ is negligible,
    	\item $T_{\lambda}|_{C}$ is projective, 
    	\item $\Phi_H(T_{\lambda})=0$, and
    	\item $\lambda\notin A_0$.
		\end{enumerate}
	\end{lemma}
	\begin{proof}
		(2)$\Rightarrow$(1) is clear, (2)$\iff$(3) follows from Lemma \ref{OTI 0 iff proj and SES}, and $(3)\iff(4)$ is exactly Corollary \ref{cor oti tilts}.  
		Finally for (1)$\Rightarrow$(4), {if $\lambda\in A_0$ then we have $T_{\lambda}=\nabla_{\lambda}$, so we may apply the Weyl dimension formula}.
	\end{proof}

   The equivalence (1)$\iff$(4) was originally studied (in a different context) in \cite{GM94,AP95}.

\subsubsection{Complexes of tilting modules}
\label{Sec: complexes of tilting modules}

	The inclusion $\Tilt\sub\Rep(G)$ induces an equivalence of triangulated categories
	\[
		K^b(\Tilt) \stackrel{\sim}{\to} D^b(\Rep(G))
	\]
	between the homotopy category of tilting modules and the derived category (see e.g.~\cite[Proposition 7.17]{riche2016geometric}).  
	Consider the Verdier quotient
	\[
		\varphi: K^b(\Tilt) \to K^b(\Tilt) / K^b(\Tilt_{neg}).
	\]
	By Lemma \ref{lemma negl tilts}, this quotient is generated (as a triangulated category) by the images of $T_{\lambda}$ for $\lambda\in \mathfrak{X}_+\cap A_0$:
	\begin{equation} 
	\label{eq:quot}
		K^b(\Tilt) / K^b(\Tilt_{neg}) 
		= 
		\langle T_\lambda \; | \; \lambda \in \mathfrak{X}_{+} \cap A_0 \rangle.
	\end{equation}

\subsubsection{Complexes of tilting modules and $\Phi_H$}
\label{sec:min and Phi}

	We now use these considerations to show that the OTI functor $\Phi_H$ factors over the quotient functor considered in the previous section. 
	This eventually leads to a transparent description of the OTI functor in terms of minimal complexes.

	We claim that we have the following diagram of categories and functors
	\[
	\begin{tikzcd}
    	& & \Rep (G) \ar[dl] \ar[d] \ar[dddr,bend left,"{\Phi_H}"',swap]
    	\\
    	K^b(\Tilt) \ar[r] \ar[dd] & D^b(\Rep (G)) \ar[d] & \Rep (C) \ar[dl] \ar[ddl,bend left] 
    	\\
    	& D^b(\Rep (C)) \ar[d] 
    	\\
    	K^b(\Tilt)/K^b(\Tilt_{neg}) \ar[r,dashed] & \Rep (C)^{st} \ar[rr] & & \operatorname{gr}_{\Zbb}\Ver_p
	\end{tikzcd}
	\]
	The rightmost commuting square is the definition of the OTI functor $\Phi_H$ {(\S\ref{sec: Principal capricorn})}. 
	The commutativity of the middle diamond and triangle were discussed in \S\ref{sec: Principal capricorn}.
	Finally, for the leftmost rectangle first note that any {bounded} complex of negligible tilting modules is perfect when restricted to $C$ by Lemma \ref{lemma negl tilts}, and is hence zero in $\Rep (C)^{st}$. 
	Thus the dashed arrow exists by the universal property of the Verdier quotient.

	With the above diagram in mind, we extend the functor $\Phi_{H}^{st}:\operatorname{Rep}(G)\to \Rep(C)^{st}$ to a functor defined on the bounded homotopy category $\Phi_{H}^{st}:K^b(\operatorname{Tilt})\to\Rep(C)^{st}$.

          \begin{remark} \label{rem:homological} We now explain that
            the restriction of this functor to the principal block is
            homological, in the sense of \cite[Definition 1.1.7]{Neeman}.
            	If we consider the restriction of $\Phi_H^{st}$ to
                $D^b(\operatorname{Rep}^{\ext}_0(G)) \cong
                K^b(\operatorname{Tilt}^{\ext}_0)$ then its image lies
                in
                $\Rep(C)^{st}_p$ by Proposition \ref{prop stable cat im}. By composing
                $\Phi_H^{st}$ with the equivalence to graded super
                vector spaces in Proposition \ref{prop:equiv to svec},
                we may regard $\Phi_H^{st}$ as taking values in
                $\operatorname{gr}_{\Zbb}\operatorname{sVec}$. One may
                then deduce that $\Phi_H^{st}$ applied to any distinguished triangle $X
                \to Y \to Z \to X[1]$, yields a long exact sequence
                \[
\dots \to \Phi_H^{st}(X)_a \to \Phi_H^{st}(Y)_a \to \Phi_H^{st}(Z)_a
\to \Phi_H^{st}(X)_{a+1}\to \dots
                  \]
                where $V_a$ denotes the $a$-th graded component of an
                object $V\in\operatorname{gr}_{\Zbb}\sVec$. Indeed, this
                follows because one may identify the functor
                $\Phi_H^{st}(-)_a$ with $\Hom_{\Rep(C)^{st}_p }(M[-a], \Phi_H^{st}(-))$, where $M =
                M(0,0) \oplus M(2-p,p-2)$.
          \end{remark}
        
\subsubsection{Rouquier complexes}
\label{Sec: Rouquier complexes}

	For any simple reflection $s \in S$, consider the 2-term complexes of functors
	\[
		F_s : \id \to \Theta_s
	\]
	and
	\[
		E_s : \Theta_s \to \id
	\]
	where in both cases the wall-crossing functor $\Theta_s$ lies in cohomological degree zero, and the nonzero differentials arise from (any choices of) adjunctions for the adjoint pairs $(\theta_{\mu_s}^{\lambda}, \theta^{\mu_s}_{\lambda})$ and $(\theta^{\mu_s}_{\lambda}, \theta_{\mu_s}^{\lambda})$. 
	Because both $\id$ and $\Theta_s$ preserve tilting modules, complexes of functors built out of $\id$ and $\Theta_s$ act on $K^b(\Tilt)$ via a double complex construction. 
	In particular, $F_s$ and $E_s$ act on $K^b(\Tilt)$.

	\begin{lemma}
    	For any reduced expression $x = s_1s_2 \dots s_m$ for $x \in {}^fW$ and $p$-regular weight $\lambda_0 \in A_0$ we have isomorphisms
    	\begin{align}
			\nabla_{x \cdot \lambda_0} &\cong F_{s_m} \dots F_{s_2}F_{s_1}(\nabla_{\lambda_0}) \label{eq:res1}\\
			\Delta_{x \cdot \lambda_0} &\cong E_{s_m} \dots E_{s_2}E_{s_1}(\nabla_{\lambda_0}) \label{eq:res2}
    	\end{align}
    	in $D^b(\Rep (G))$.
	\end{lemma}

	\begin{proof} 
		See \cite[Lemma 2.1]{LibWil} or the proof of \cite[Proposition 2.4]{Gruber}.
	\end{proof}

	\begin{remark}
		Because $\nabla_{\lambda_0} = \Delta_{\lambda_0}$ is tilting (as $\lambda_0$ belongs to the fundamental alcove), the complexes in the lemma provide tilting resolutions of standard and costandard modules. 
		These are typically not minimal complexes.
	\end{remark}

	Note that $\Theta_s=\theta_{\mu_s}^{\lambda}\circ\theta_{\lambda}^{\mu_s}$ maps $\Tilt$ into $\Tilt_{neg}$ since $\theta_{\lambda}^{\mu_s}$ clearly does and $\Tilt_{neg}$ is a tensor ideal. 
	In particular, on the quotient
	\[
		K^b(\Tilt)/K^b(\Tilt_{neg})
	\]
	the complexes $F_s$ (resp. $E_s$) act as $[1]$ (resp. $[-1]$). 
	We deduce isomorphisms (for $x$ as in the lemma):
 	\begin{align}
    	\Phi_H^{st}(\nabla_{x \cdot \lambda_0}) & \cong \Phi_H^{st}(\nabla_{\lambda_0})[\ell(x)] \label{eq:Phi1} 
    	\\
    	\Phi_H^{st}(\Delta_{x \cdot \lambda_0}) & \cong \Phi_H^{st}(\Delta_{\lambda_0})[-\ell(x)] \label{eq:Phi2}
	\end{align}
	where $\Phi_H^{st}$ is defined in (\ref{eqn Phi st}).

	\begin{remark}
    	Equations \eqref{eq:Phi1} and \eqref{eq:Phi2} along with Lemma \ref{lemma shift Pi} immediately imply an alternative proof of Theorem A(1) when $H\cong\alpha_p$. 
	\end{remark}

\subsubsection{Minimal tilting complexes and $\Phi_H$}
\label{sec:minPhi}
	Let $M \in K^b(\Tilt)$ be a minimal complex of tilting modules. 
	For each $i$ we fix decompositions
	\[
		M^i = \bigoplus T_\lambda^{\oplus m_{i,\lambda}}
	\]
	The aim of this section is to prove: 
	\begin{thm}  
	\label{thm:minimal}
		For any minimal complex $M \in K^b(\Tilt)$ as above,
    	\[ 
    		\Phi_H^{st}(M) 
    		\cong 
    		\bigoplus_{\substack{\lambda \in \mathfrak{X}_+ \cap A_0\\ {i \in \Zbb}}} \Phi_H^{st}(T_{\lambda})[-i]^{\bigoplus m_{i,\lambda}}. 
    	\]
	\end{thm}
	In other words, we can compute $\Phi_H^{st}$ (and hence $\Phi_H$) by simply discarding every negligible summand from our complex $M$, evaluating $\Phi_H^{st}$ term by term, and adding up the result.

	\begin{remark}
    	Theorem \ref{thm:minimal} can be used to give an alternative route to our main Theorem B from the introduction. 
    	Indeed, one just needs to know that $\Phi_H$ maps the tilting modules $T_\lambda$ for $\lambda \in \Omega \cdot 0$ corresponding to the highest weights of non-negligible tiltings in the extended principal block to $\sVec$, which follows from Lemma \ref{lemma prep}(3).
	\end{remark}

	It is easy to see that one may prove Theorem \ref{thm:minimal} ``block by block".
	We will prove the result for the principal block, with the other blocks following in a similar way. 
	To this end, let $M$ be a minimal complex of tilting modules in $\operatorname{Tilt}_0$. 
	For each $i$ we fix decompositions
	\[
		M^i = \bigoplus_{x \in {}^fW} T_{x \cdot 0}^{\oplus m_{i,x}}
	\]
	as a direct sum of indecomposable tilting modules. 
	The version of Theorem \ref{thm:minimal} for the principal block is then the following:

	\begin{thm}  
	\label{thm:minimal0}
		For any minimal complex $M\in K^b(\Tilt_0)$ as above
    	\[ 
    		\Phi_H^{st}(M) 
    		\cong 
    		\bigoplus_{{i \in \Zbb}} \Phi_H^{st}(T_0)[-i]^{\bigoplus m_{i,\id}}. 
    	\]
	\end{thm}
	In other words, $\Phi_H^{st}$ (and hence $\Phi_H$) detects exactly those summands of the minimal complex which are isomorphic to $T_0$.

	Throughout, we will use the interpretation of the OTI functor given in \S\ref{sec:min and Phi}. 
	In other words, $\Phi_H^{st}$ is viewed as a functor
	\[
		\Phi_H^{st}: K^b(\Tilt_0)/K^b(\Tilt_{0,neg}) \to \Rep (C)^{st}
	\]
	where $\Tilt_{0,neg}$ denotes the additive category of negligible tilting modules in the principal block.

	Let $M$ be as in the statement of Theorem \ref{thm:minimal0}. 
	Denote the minimal and maximal degrees of $M$ by $m$ and $n$. 
	Thus $M$ has the form
	\[
		0 \to M^m \to M^{m+1} \to \dots \to M^n \to 0.
	\]
	For every $a$ we can consider the subcomplex $M^{\ge a}$ which is zero in degrees $<a$ and has the terms of $M$ in degrees $\ge a$. 
	We have a diagram of triangles, expressing $M$ as an iterated extension of its terms:
	\begin{equation} 
	\label{eq:stupid filt}
    	\begin{tikzcd}[column sep=0.2cm]
			0 \arrow[rr] & &  M^{\ge n} \arrow[rr] \arrow[dl] & & M^{\ge n-1} \arrow[rr] \arrow[dl] & &  \dots \arrow[rr] & & M^{\ge m} \arrow[dl] 
			\\
			& M^n[-n] \arrow[ul, "{[1]}"', swap] & & M^{n-1}[-(n-1)] \arrow[ul, "{[1]}"', swap] & & \dots  & &M^m[-m] \arrow[ul, "{[1]}"', swap]
		\end{tikzcd}
	\end{equation}
	(This is often referred to as the ``stupid filtration'' of $M$.)

	\begin{lemma} 
	\label{lem:diff0}
    	The image of the differential $d : M^a \to M^{a+1}$ under $\Phi_H^{st}$ is zero.
	\end{lemma}

	\begin{proof} 
		Recall our decompositions 
		\[
			M^a = \bigoplus\limits_{x \in {}^fW} T_{x \cdot0}^{\oplus m_{a,x}}
			\quad \text{and} \quad
			M^{a+1} = \bigoplus\limits_{x \in {}^fW} T_{x \cdot0}^{\oplus m_{a+1,x}}.
		\]
		Let us regard our differential $d$ as a matrix of morphisms between these indecomposable tilting modules. 
		Because $\Phi_H^{st}$ kills any negligible tilting module we only need to check entries of our matrix whose source and target are $T_0$. 
		However, $T_0$ is simple, so we only need to check that no entries of our matrix consist of scalar multiples of the identity map from $T_0$ to $T_0$. 
		This is implied by our assumption that our complex is minimal.
	\end{proof}

	We are now ready to prove Theorem \ref{thm:minimal}:

	\begin{proof}
    	We prove
    	\begin{equation} 
    	\label{eq:ind}
    		\Phi_H^{st}(M^{\ge j}) 
    		= 
    		\bigoplus_{i \ge j} \Phi_H^{st}(T_0)[-i]^{\bigoplus m_{i,\id}}
   			= 
   			\bigoplus_{i \ge j} M(0,0)[-i]^{\bigoplus m_{i,\id}} 
    	\end{equation}
    	by descending induction on $j$, with the base case $j = n+1$ being trivial and the final case $j = m$ being the statement of Theorem \ref{thm:minimal}. 
    	(Recall that $M(0,0)$ denotes the image of the trivial module in the stable category of $C$-modules, as in \S\ref{sec: Capricorn}.)

		Consider the distinguished triangle
		\[
			M^{\ge j+1} \to M^{\ge j} \to M^j[-j] \stackrel{[1]}{\to} .
		\]
		We are done if we can show that $\Phi_H^{st}$ sends the boundary map in this triangle
		\begin{equation} 
		\label{eq:boundary}
			M^j \to M^{\ge j+1}[1]
		\end{equation}
		to zero. 
		To this end, note that the distinguished triangle 
		\[
			M^{\ge j+2} \to M^{\ge j+1} \to M^{j+1}[-j-1] \stackrel{[1]}{\to} 
		\]
		yields (after applying $[1]$ and taking hom from $M^j[-j]$) a long exact sequence
		\[
			\dots 
			\to \Hom(M^j[-j], M^{\ge{j+2}}[1])
			\to \Hom(M^j[-j], M^{\ge{j+1}}[1])
			\to \Hom(M^j[-j], M^{j+1}[-j])
			\to \dots
		\]
		Applying $\Phi_H^{st}$ yields a commutative diagram of long exact sequences:
		\begin{equation} 
		\label{eq:big commute}
    		\begin{tikzcd}[cramped, row sep=0.8em, column sep=0.6em]
				\dots \arrow[r] &  \Hom(M^j[-j], M^{\ge j+2}[1]) \arrow[r] \arrow[d] & \Hom(M^j[-j], M^{\ge j+1}[1]) \arrow[r] \arrow[d] & \Hom(M^j[-j], M^{j+1}[-j]) \arrow[r] \arrow[d] & \dots \\
				\dots \arrow[r] &  \Hom(\overline{M^j[-j]}, \overline{M^{\ge j+2}}[1]) \arrow[r]  & \Hom(\overline{M^j[-j]}, \overline{M^{\ge j+1 }}[1]) \ar[r] &  \Hom(\overline{M^j[-j]}, \overline{M^{j+1}}[-j]) \arrow[r] & \dots 
			\end{tikzcd}
		\end{equation}
		(For display purposes only we use an overline to indicate the application of the functor $\Phi_H^{st}$ -- so for example, $\overline{M^j} = \Phi_H^{st}(M^j)$ etc.).

		Our boundary map in \eqref{eq:boundary} gives rise to an element $b \in \Hom(M^j[-j], M^{\ge{j+1}}[1])$ which we claim goes to zero under $\Phi_H^{st}$; {i.e., the middle vertical arrow in \eqref{eq:big commute}}.

		We first claim that 
		\[
			\Hom(\Phi_H^{st}({M^j[-j]}), \Phi_H^{st}(M^{\ge{j+2}}[1])) 
			= 
			0.
		\]
		Indeed, the left hand side is isomorphic to a direct sum of copies of $M(0,0)[-j]$, whilst the right hand side is isomorphic to a direct sum of copies of $M(0,0)[-k]$ with $k > j$ by induction. 
		The claimed vanishing follows from Lemma \ref{lemma homs st cat}.

		Thus the lower left group in \eqref{eq:big commute} is zero, and in order to check that our boundary map goes to zero it is enough to show that its image in the lower right group is zero. 
		By the commutativity of the right-hand square of \eqref{eq:big commute}, its image there agrees with the image of the differential in $\Hom(M^j[-j], M^{j+1}[-j])$. 
		However, this differential goes to zero under $\Phi_H^{st}$ by Lemma \ref{lem:diff0}. 
		We conclude that $b$ indeed goes to zero under $\Phi_H^{st}$ which concludes the proof. 
	\end{proof}

\subsection{Singular modules and $\Phi_H$} 
\label{sec: Singular modules}

	Recall from the introduction that a $G$-module is called singular if its minimal tilting resolution involves only negligible tilting modules. 
	The goal of this section is to prove:

	\begin{prop} 
	\label{prop:detectsGruber}
		A $G$-module $V$ is singular if and only if $\Phi_H(V) = 0$.
	\end{prop}

	\begin{remark}
    	Recall Gruber's quotient category
    	\[
    		\underline{\Rep}(G) = \Rep (G)/\Rep (G)_{sing}
    	\]
    	from the introduction. 
    	The proposition shows that the OTI functor faithfully detects objects in $\underline{\Rep}(G)$
	\end{remark}

	\begin{proof} 
		Let us choose a minimal tilting complex $M \in K^b(\Tilt)$ which is quasi-isomorphic to $V$. 
		By definition, $V$ is singular if $M$ consists entirely of negligible tilting modules; in other words, if no summand of any term of $M^i$ is isomorphic to  $T_{\lambda_0}$, where $\lambda_0 \in A_0$. 
		By Theorem \ref{thm:minimal} this is the case if and only if $\Phi_H(V) = 0$.
	\end{proof}

	\begin{remark}
    	The above proof benefited from discussions with Jonathan Gruber.
	\end{remark}

\subsection{The Finkelberg-Mirkovi\'c conjecture and $\Phi_H$}
\label{sec: FM Conj and Phi}

	In this section we discuss our functor in relation to the Finkelberg-Mirkovi\'c conjecture. 
	As discussed in the introduction, the Finkelberg-Mirkovi\'c conjecture is now a theorem \cite{BRR, BR1, BR2}. 
	Here we discuss the conjecture (stated as Conjecture \ref{conj:FM} in the introduction) that our functor is isomorphic to hypercohomology under this equivalence.

\subsubsection{The geometric Satake equivalence and hypercohomology}
\label{sec: Geometric Satake}

	Let $G$ be as above and let ${}^LG$ be the complex group which is dual in the sense of Langlands to $G$.  
	Let $\Gr= {}^LG((t))/{}^LG[[t]]$ denote the affine Grassmannian for ${}^LG$. 
	The geometric Satake equivalence \cite{MV} provides an equivalence of (symmetric) tensor categories
	\[
		(\Rep (G), \otimes) \stackrel{\sim}{\to} ( P_{{}^LG[[t]]}(\Gr,\Bbbk), \star),
	\]
	where $P_{{}^LG[[t]]}(\Gr,\Bbbk)$ denotes the category of perverse sheaves with $\Bbbk$ coefficients, which are equivariant with respect to the left action of ${}^LG[[t]]$, and $\star$ is the convolution product. 
	This provides a geometric realization of $\Rep (G)$, but the block decomposition is opaque (see, however, \cite{RWlinkage}).

	On $\Rep (G)$ one has the forgetful functor to vector spaces, which is a tensor functor, and from which one can recover $G$ via Tannakian formalism. 
	A key step in the proof of geometric Satake is to prove that hypercohomology on $P_{{}^LG[[t]]}(\Gr)$ is a faithful tensor functor, and thus provides a fibre functor.

	It was noticed by Ginzburg \cite{GinzburgLoop} that the hypercohomology functor in geometric Satake has more structure. 
	(Ginzburg assumes that $\Bbbk$ is of characteristic $0$, and we will do the same for this paragraph.) 
	Namely, if one thinks of hypercohomology as homomorphisms from the constant sheaf $\Bbbk_{\Gr}$ then it is clear that it results in graded modules over the derived endomorphisms of $\Bbbk_{\Gr}$, which is the cohomology ring of $\Gr$:
	\[
		H^* 
		= 
		\Hom^{*}(\Bbbk_{\Gr},-): 
		P_{{}^LG[[t]]}(\Gr,\Bbbk) \to H^*(\Gr,\Bbbk)\text{-}\operatorname{grMod}.
	\]
	Ginzburg goes on to prove that $H^*(\Gr,\Bbbk)$ is naturally a Hopf algebra, isomorphic to the enveloping algebra of the centraliser $\mathfrak{g}^e$ of a regular nilpotent element $e \in \mathfrak{g}$, where $\mathfrak{g}$ is the Lie algebra of $G$:
	\[
		H^*(\Gr,\Bbbk) 
		= 
		U(\mathfrak{g}^e).
	\]
	Moreover, under the geometric Satake equivalence and this
        isomorphism, restriction to $\mathfrak{g}^e$ is isomorphic to
        the hypercohomology functor:
	\begin{equation} 
	\label{eq:ginzburg diagram}
		\begin{tikzcd}
    		\mathrm{Rep}(G) \arrow[r, "\sim"] \arrow[rd, swap, "res"] & P_{{}^LG[[t]]}(\Gr, \Bbbk) \arrow[d, "H^*"] \\
    		& H^*(\mathrm{Gr},\Bbbk) \text{-}\operatorname{grMod}
		\end{tikzcd}
	\end{equation}

	The setting when $\Bbbk$ is of positive characteristic is more
        subtle, and has been worked out by Yun and Zhu \cite{YunZhu}. 
	In this case, one has an isomorphism \cite[Corollary 6.4]{YunZhu}
	\begin{equation} 
	\label{eq:YZ}
		H^*(\Gr,\Bbbk) 
		= 
		\Dist C_G(e),
	\end{equation}
	where $e$ is a regular unipotent element, $C_G(e)$ its centraliser group scheme in $G$, and $\Dist$ denotes its distribution algebra.\footnote{\cite{YunZhu} prove such an isomorphism under certain restrictions on $p$, however $p \ge h$ is always enough. We are also brushing some connectedness issues under the rug to simplify notation.} 
	Moreover, under this identification the obvious analogue of the diagram \eqref{eq:ginzburg diagram} commutes.

\subsubsection{The Finkelberg-Mirkovi\'c conjecture}
\label{sec: FM conj}

	As we discussed above, one drawback of geometric Satake for attacking questions in the representation theory of $G$ geometrically is that it does not see the block decomposition. 
	The Finkelberg-Mirkovi\'c conjecture addresses this defect. 
	As in the introduction, let $\mathrm{Iw}$ denote the Iwahori subgroup of ${}^LG((t))$ corresponding to our choice of Borel ${}^LB \subset {}^LG$.
	Finkelberg and Mirkovi\'c conjectured an equivalence
	\begin{equation} 
	\label{eq:FM}
		\Rep_0^{\ext}(G) \stackrel{\sim}{\to} P_{(\mathrm{Iw})}(\Gr,\Bbbk),
	\end{equation}
	where $P_{(\mathrm{Iw})}(\Gr,\Bbbk)$ denotes the category of perverse sheaves which are constructible with respect to the stratification by Iwahori orbits.

	We now discuss Conjecture \ref{conj:FM}.  
	Recall from Proposition \ref{prop stable cat im} that $\Phi_H^{st}$ takes values in $\Rep(C)^{st}_p$ which, by Proposition \ref{prop:equiv to svec}, is equivalent to graded super vector spaces. 
    Thus $\Phi_H$ can be seen as taking values in super vector spaces. We have also seen that $\Dist C_G(e)^{(1)}$ acts naturally on $\Phi^{st}_H$. 
	Our first goal is to prove that these two actions are compatible. 
	More precisely, that $\Phi_H$ may be regarded as taking values in $\Dist C_G(e) \, \text{-}\operatorname{grMod}$, the category of graded $\Dist C_G(e)$-modules.

	First note that we have a canonical isomorphism\footnote{As
          pointed out by the referee, one can get rid of the need for
          this identification if, in the Finkelberg-Mirkovi\'c conjecture, one considers the affine Grassmannian of a group whose Langlands dual is $G^{(1)}$. (This is the formulation that is actually proved in \cite{BR2}.)}
	\[
		\Dist C_G(e)^{(1)} \cong \Dist C_G(e)
	\]
	that divides degrees by $p$. 
	One way to get this isomorphism is to note that \eqref{eq:YZ} holds over any field, in particular over $\mathbb{F}_p$, which yields an $\mathbb{F}_p$-rational structure on $C_G(e)$.

	\begin{lemma}
  		For any $V \in \Rep_0^{\ext}(G)$, $\Phi_H(V)$ is naturally a graded $\Dist C_G(e)$-module.
	\end{lemma}

	\begin{proof} 
  		Fix $\gamma \in \Dist C_G(e)$ of degree $m$, and let $\gamma' \in \Dist C_G(e)^{(1)}$ denote the corresponding element of $\Dist C_G(e)^{(1)}$ of degree $pm$.  
  		This element acts naturally on $V$, and hence on $\Phi_H^{st}(V)$, as an endomorphism of degree $pm$:
  		\begin{equation} 
  		\label{eq:gammas home}
			\gamma' \in \Hom(\Phi^{st}_H(V), \Phi^{st}_H(V) \otimes M(pm, 0)).
		\end{equation}
  		We may assume that $V$ is indecomposable, and
                therefore without loss of generality, that
                $\Phi^{st}_H (V) \in \langle
                M(a,0) \rangle \subset \Rep
                (C)^{st}$, where $a \in \{ 0,-p\}$.
  		Now \eqref{eq:gammas home} is zero unless $m$ is even, because $p$ is odd.\footnote{Alternatively, we could have used that the grading on $\Dist C_G(e)$ vanishes in odd degree, which can be deduced from \eqref{eq:YZ}.}
  		Hence, $\gamma'$ lives in $\Hom(\Phi^{st}_H(V), \Phi^{st}_H(V)[m])$ by Lemma \ref{lemma shifts of O}. 
  		In particular, after applying our equivalence $\mathcal{E}$, we deduce from Proposition \ref{prop:equiv to svec} that $\gamma'$ induces an operator of degree $m$. 
  		This is our natural action of $\gamma$.
	\end{proof}

    It follows that we may regard $\Phi_H$ as taking values in $\Dist C_G(e)$-$\operatorname{grMod}$. 
    The following connects this action to the Finkelberg-Mirkovi\'c conjecture:

	\begin{conj} 
	\label{conj:FM2}
		We have a commuting diagram:
		\[
		\begin{tikzcd}
    		\mathrm{Rep}_0^{\mathrm{ext}}(G) \arrow[r, "\sim"] \arrow[d, swap, "\Phi_H"] & P_{(\mathrm{Iw})}(\mathrm{Gr}, \Bbbk) \arrow[d, "H^*"] \\
  			\Dist C_G(e)\text{-}\operatorname{grMod} \arrow[r, "\sim"]  & H^*(\mathrm{Gr})\text{-}\operatorname{grMod}
		\end{tikzcd}
              \]
              (Recall that $A\text{-}\operatorname{grMod}$ denotes graded
                modules over the graded ring $A$.)
	\end{conj}

	In order to check that our conjecture is plausible, we check that $\Phi_H$ and $H^*$ take the same values on standard, costandard and tilting objects in the principal block. 
	For any $x \in {}^fW$, we denote by
	\[
	j_x : X_x = \mathrm{Iw} \cdot x^{-1} \; {}^LG[[t]]/{}^LG[[t]] \hookrightarrow \Gr
	\]
	the inclusion of the Schubert cell indexed by $x^{-1}$ into the affine Grassmannian. 
	Under the Finkelberg-Mirkovi\'c conjecture one has: 
	\begin{align*}
		\nabla_{x \cdot 0} & \mapsto j_{x*}\Bbbk_{X_x}[\ell(x)], \\
		\Delta_{x \cdot 0} & \mapsto j_{x!}\Bbbk_{X_x}[\ell(x)], \\
		T_{x \cdot 0} & \mapsto \mathcal{T}_x
	\end{align*}
	where $\mathcal{T}_x$ denotes the indecomposable tilting perverse sheaf with support $\overline{X_x}$.

	Because $x \in W$, we have seen in \eqref{eq:Phi1} and \eqref{eq:Phi2} that 
	\[
		\Phi_H^{st}(\nabla_{x \cdot 0}) 
		= 
		\Phi_H^{st}(\Bbbk)[\ell(x)] 
		\quad \text{and} \quad 
		\Phi_H^{st}(\Delta_{x \cdot 0}) 
		= 
		\Phi_H^{st}(\Bbbk)[-\ell(x)].
	\]
	On the geometric side one calculates easily (using that cells are contractible):
	\[
		H^*(j_{x*} \Bbbk_{X_x}[\ell(x)]) 
		= 
		\Bbbk[\ell(x)] 
		\quad \text{and} \quad 
		H^*(j_{x!} \Bbbk_{X_x}[\ell(x)]) 
		= 
		\Bbbk[-\ell(x)].
	\]
	For tilting modules, we have (see Lemma \ref{lemma tiltings sl2})
	\[
		\Phi_H(T_{x \cdot 0}) 
		= 
		\begin{cases} 
			\Phi_H(T_0) & \text{if $x = \id$,} 
			\\ 
			0 & \text{otherwise}. 
		\end{cases}
	\]
	For tilting sheaves, one has (see \cite{tiltingexercises})
	\[
		H^*(\mathcal{T}_x) 
		= 
		\begin{cases} 
			\Bbbk & \text{if $x = \id$,} 
			\\ 
			0 & \text{otherwise}. 
		\end{cases}
	\]

	\begin{remark} 
		We finish with several remarks concerning our conjecture:
	\begin{enumerate}
    	\item The appearance of super vector spaces in the image of
          $\Phi_H$ appears natural from the algebraic side, but is
          still somewhat mysterious on the geometric side. For
          example, consider the principal block of $SL_2$ and let $0$
          and $\lambda_0$ denote the two weights in the extended
          principal block in the dominant alcove. On the algebraic
          side, $\Phi_H$ maps $L_0$ and $L_{\lambda_0}$ to
          one-dimensional even and odd vector spaces
          respectively. Geometrically, $L_0$ and $L_{\lambda_0}$ are
          realised as skyscraper sheaves on the two components of
          $\Gr$.  In a different vein, for any $G$ one may use
          Remark \ref{rem:homological} to show that $\Phi_H(L_{s\cdot 0})\cong\Pi\Bbbk^2$ for $s$ the simple affine reflection, and we note this is purely odd. We expect parity vanishing properties of $\Phi_H(L_{x\cdot 0})$ to be connected to deep questions in representation theory.  
    	\item If one interprets the functors $\Phi_H$ and $H^*$ in terms of minimal complexes of tilting modules (see \S \ref{sec:minPhi}), they have almost identical descriptions. It is likely that this should allow one to establish our conjecture if one regards both functors as landing in the derived category of vector spaces (i.e. one ignores the action of $C_G(e)^{(1)}$). In particular, one can use this observation to prove that $H^*$ and $\Phi_H$ produce isomorphic graded vector spaces when applied to simple modules.

\end{enumerate}
\end{remark}

\bibliographystyle{alpha}

\newcommand{\etalchar}[1]{$^{#1}$}

\end{document}